\documentclass[12pt, reqno]{amsart}
\usepackage{amsmath, amsthm, amscd, amsfonts, amssymb, graphicx, color}
\usepackage[bookmarksnumbered, colorlinks, plainpages]{hyperref}
\input{mathrsfs.sty}
\newtheorem{theorem}{Theorem}[section]
\newtheorem{lemma}[theorem]{Lemma}

\newtheorem{corollary}[theorem]{Corollary}
\theoremstyle{definition}
\newtheorem{definition}[theorem]{Definition}
\newtheorem{example}[theorem]{Example}

\theoremstyle{remark}
\newtheorem{remark}[theorem]{Remark}
\numberwithin{equation}{section}
\begin{scriptsize}
{\footnotesize}
\end{scriptsize}
\begin{document}
\setcounter{page}{1}

 \title[Decomposition of tracial positive maps]{Decomposition of tracial positive maps and   applications in quantum information}

\author[A. Dadkhah, M. Kian, M.S. Moslehian]{ Ali Dadkhah$^{1}$, Mohsen Kian$^{*2}$ \MakeLowercase{and} Mohammad Sal Moslehian$^{1}$}

\address{$1$\ Department of Pure Mathematics, Center Of Excellence in Analysis on Algebraic Structures (CEAAS), Ferdowsi University of Mashhad, P. O. Box 1159, Mashhad 91775, Iran}
\email{dadkhah61@yahoo.com}
\email{moslehian@um.ac.ir}

\address{$2$\ Department of Mathematics, University of Bojnord, P.O. Box 1339, Bojnord
94531, Iran}
\email{kian@ub.ac.ir}

\renewcommand{\subjclassname}{\textup{2020} Mathematics Subject Classification}


\maketitle
\begin{abstract}
Every positive multilinear map between $C^*$-algebras is separately weak$^*$-continuous. We show that the joint weak$^*$-continuity is equivalent to the joint weak$^*$-continuity of the multiplications of $C^*$-algebras under consideration. We study the behavior of general tracial positive maps on properly infinite von Neumann algebras and by applying the Aron--Berner extension of multilinear maps, we establish that under some mild conditions every tracial positive multilinear map between general $C^*$-algebras enjoys a decomposition $\Phi=\varphi_2 \circ \varphi_1$, in which $\varphi_1$ is a tracial positive linear map with the commutative range and $\varphi_2$ is a tracial completely positive map with the commutative domain. As an immediate consequence, tracial positive multilinear maps are completely positive. Furthermore, we prove that if the domain of a general tracial completely positive map $\Phi$ between $C^*$-algebra is a von Neumann algebra, then $\Phi$ has a similar decomposition.
As an application, we investigate the generalized variance and covariance in quantum mechanics via arbitrary positive maps. Among others, an uncertainty relation inequality for commuting observables in a composite physical system is presented.
\end{abstract}

\section{Introduction and Preliminaries}
Let $\mathscr{A}$ be a unital $C^*$-algebra with the unit $I$. A self-adjoint subspace $\mathcal{S}\subseteq \mathscr{A}$ containing $I$ is called an operator system. We use $\mathcal{S}^+$ and $\mathcal{S}^{++}$ to denote the sets of all positive and positive invertible elements of $\mathcal{S}$, respectively.
For $1\leq i\leq k$, let $\mathscr{A}_i$ be some unital $C^*$-algebra
 with the unit $I_i$. Then $\bigoplus_{i=1}^k \mathscr{A}_i $ stands for the direct sum of $\mathscr{A}_i$'s, which is a unital $C^*$-algebra with the unit $\textbf{I}=(I_1,\ldots, I_k)$. In the case when $\mathscr{A}_i=\mathscr{A}$ for all $1\leq i\leq k$, we simply denote $\bigoplus_{i=1}^k \mathscr{A}_i $ by $\mathscr{A}^{(k)}$ and use the notation $A^{(m)}$ for $m$-tuple $(A, \ldots, A)$. More generally, for a $C^*$-algebra $\mathscr{A}$, the direct sum $\bigoplus_{i=1}^\infty \mathscr{A}$, denoted by $\mathscr{A}^{\oplus_{i=1}^\infty}$, is the $C^*$-algebra of all $(A_1, A_2, \ldots)$ such that $\|(A_1, A_2, \ldots)\|:= \sup_i \|{A_i}\|$ is finite. If $\mathcal{S}_i\subseteq \mathscr{A}_i$ are some operator systems in the $C^*$-algebras $\mathscr{A}_i$ for $ 1\leq i \leq k$, then the direct sum $\bigoplus_{i=1}^k \mathcal{S}_i $ is an operator system in $\bigoplus_{i=1}^k\mathscr{A}_i $.
 Throughout the paper, $\mathbb{B}(\mathscr{H})$ stands for the $C^*$-algebra of all bounded linear operators on a complex Hilbert space $(\mathscr{H}, \langle\cdot,\cdot\rangle)$ with the unit $I$. Due to the Gelfand--Naimark--Segal theorem, any $C^*$-algebra $\mathscr{A}$ can be considered as a closed $*$-subalgebra of $\mathbb{B}(\mathscr{H})$ for some Hilbert space $\mathscr{H}$, and we say that $ \mathscr{A}$ acts on $\mathscr{H}$.\\
 A $C^*$-algebra $\mathscr{B}$ is called injective if for every unital $C^*$-algebra $\mathscr{A}$, every operator system $\mathcal{S}\subseteq \mathscr{A}$ and each completely positive linear map $\varphi: \mathcal{S} \to \mathscr{B}$, there exists a completely positive linear map
$\tilde{\varphi}: \mathscr{A} \to \mathscr{B}$ such that $\tilde{\varphi}|_\mathcal{S} =\varphi $ and $\|\tilde{\varphi}\|=\| \varphi\|$. In~particular, Arveson's extension theorem ensures that $\mathbb{B}(\mathscr{H})$ is injective for any Hilbert space $\mathscr{H}$; see \cite[Theorem 7.5]{Paulbook}. It is also known that for a compact Hausdorff space $\Omega$, the $C^*$-algebra $C(\Omega)$ is injective if and only if $\Omega$ is a Stonean space. \\
Two projections $P$ and $Q$ in a von Neumann algebra $\mathscr{M}$ are called equivalent if there exists an element $U\in \mathscr{M}$ such that
$P=U^*U$ and $Q=UU^*$. We write $P\sim Q$ when $P$ and $Q$ are equivalent. A projection $P$ is called
finite if $P\sim Q$ and $P\leq Q$ implies $P=Q$. Otherwise, $P$ is said to be infinite. A nonzero projection $P$ is called properly infinite if for every nonzero central projection $E\in \mathscr{M}$, the projection $EP$ is infinite. A von Neumann algebra is said to be finite or properly infinite if its identity has the corresponding property; see \cite[Chapter V, Definitions 1.15 and 1.16]{TAK1}.
 According to \cite[Chapter V, Proposition 1.23 (ii)]{TAK1}, if $\mathscr{M}$ is a commutative von Neumann algebra, then $M_n(\mathscr{M})$, the algebra of all $n\times n$ matrices with entries in $\mathscr{M}$, is a finite von Neumann algebra. In particular, $M_n(\mathbb{C})$ is a finite von Neumann algebra. \\
Let $\mathscr{B}$ be a unital $C^*$-algebra with the unit $I$. Then a multimap $\Phi : \bigoplus_{i=1}^k \mathcal{S}_i \to \mathscr{B} $ is called unital if $\Phi(I_1,\ldots,I_k) =I$ and is said to be positive when $\Phi\left(\bigoplus_{i=1}^k \mathcal{S}_i^+\right)\subseteq \mathscr{B}^+$. We say that $\Phi$ is tracial if $\Phi(\textbf{AB})=\Phi(\textbf{BA})$ for every $\textbf{A,B} \in \bigoplus_{i=1}^k \mathcal{S}_i$. A map $\Phi : \bigoplus_{i=1}^k \mathcal{S}_i \to \mathscr{B} $ is called multilinear if it is a $k$-linear map in the sense that $\Phi ( A_1, \ldots, \lambda A_j+B,\ldots, A_k )= \lambda \Phi ( A_1, \ldots, A_j,\ldots, A_k )+\Phi ( A_1, \ldots, B,\ldots, A_k )$ for all $\lambda \in \mathbb{C}$ and $1\leq j\leq k$. A map $\Phi : \mathscr{A} \to \mathscr{B}$ is said to be $(m, n)$-mixed homogeneous for some nonnegative integers $m$ and $n$ if $\Phi(zA)= z^m \bar{z}^n\Phi(A)$ for all $z\in \mathbb{C}$ and $A\in \mathscr{A}$. Throughout the paper, when a map or a multimap is not necessarily linear, multilinear, or $(m, n)$-mixed homogeneous, we write them simply map or multimap, respectively.

 The following lemma is useful in our work, which is an easy consequence of \cite[Theorem 1.3.3]{bht}; see also \cite{MKX}.
 \vspace{-.3cm}
\begin{lemma}\label{epsilpo}
 Let $A$ and $B$ be a positive element of a unital $C^*$-algebra. Then $\begin{bmatrix}
A & X \\ X^* & B
\end{bmatrix}\geq 0 $ if and only if $B\geq X^* (A+\varepsilon I)^{-1} X$ for all $\varepsilon>0$.
\end{lemma}
Hence, the positivity of the operator matrix $\begin{bmatrix} A& X\\ X^*& 0 \end{bmatrix}$ entails the vanishing of $X$. If ${A}$ is an $n\times n $ matrix and ${B}$ is an $m\times m$ matrix, then ${A}\oplus {B}:= \begin{bmatrix}
 {A} & \textbf{0}_{n\times m } \\ \textbf{0}_{m\times n} & {B}
 \end{bmatrix}$. It is easy to check that
 \[ {A}, {B} \geq 0 \Longleftrightarrow {A}\oplus {B} \geq 0.\] We use $X \circ Y$ to denote the Schur (Hadamard) product of $X$ and $Y$. According to \cite[Proposition 12]{shon} if the domain of a unital multilinear map $\Phi$ is a commutative $C^*$-algebra, then $\Phi$ is completely positive. In particular, if $\mathscr{A}$ is a unital commutative $C^*$-algebra, then the positive bilinear map $\pi : \mathscr{A} \bigoplus \mathscr{A} \to \mathscr{A}$ given by $\pi(A,B)=AB$ is completely positive. Therefore, the Schur product of each pair of positive matrices in $M_n(\mathscr{A})$ ($n\in \mathbb{N}$) is positive.
\subsection*{Complete Boundedness and $n$-positivity of Multimaps}
There are at least two known ways to define the complete boundedness and $n$-positivity of a multimap $\Phi : \bigoplus_{i=1}^k \mathcal{S}_i \to \mathscr{B} $. \\
\textbf{Type 1.} The first definition was presented in the works of Effros \cite{effros}, Christensen and Sinclair \cite{crist2, crist}, and Paulsen and Smith \cite{smith} as follows. To define the complete boundedness and the complete positivity, the definition is given only in the case when all domain $C^*$-algebras are the same. However, it can be extended to the general case; see \cite[Sec. 2]{joita1} and \cite[Sec. Introduction and 1.1]{crist}. For every $n\in \mathbb{N}$, a (multilinear) map $\Phi : \mathscr{A}^{(k)} \to \mathscr{B} $ induces a (multilinear) map $\Phi^n : M_n(\mathscr{A})^{(k)} \to M_n(\mathscr{B})$ by
\[ \Phi^n \left(A_1,\ldots, A_k \right)=\left[\sum_{l,r, s,\ldots,t=1}^n \Phi(A_{1_{il}}, A_{2_{lr}},\ldots, A_{k_{tj}})\right]_{i,j=1}^n,
\]
where $A_l=[A_{l_{ij}}] \in M_n(\mathscr{A})$, $l=1, \ldots, k$.
In the multilinear case, a map $\Phi : \bigoplus_{i=1}^k \mathscr{A}_i \to \mathscr{B} $ is called completely bounded if $\| \Phi\|_{cb}= \sup \{ \| \Phi^n\| : n \in \mathbb{N}\} <\infty$. If $k$ is an even natural number, then a multimap $\Phi: \mathscr{A}^{(k)} \to \mathscr{B}$ is called $n$-positive if $\Phi^n(A_1, \ldots, A_k) \geq 0$ for all $(A_1, \ldots, A_k)=(A_k^*, \ldots, A_1^*) \in M_n(\mathscr{A})^{(k)}$. In the case when $k$ is odd, $\Phi$ is said to be $n$-positive if $\Phi^n(A_1, \ldots, A_k) \geq 0$ for all $(A_1, \ldots, A_k)=(A_k^*, \ldots, A_1^*) \in M_n(\mathscr{A})^{(k)}$ with $A_m \geq 0$, where $A_m$ is the midpoint of $k$-tuple $(A_1, \ldots, A_k)$.
The map $\Phi$ is said to be completely positive of type 1 if $\Phi^n$ is positive for all $n\in\mathbb{N}$.\\
As an example, for any $C^*$-algebra $\mathscr{A}$, the map $\Pi : \mathscr{A}^{(k)} \to \mathscr{A}$ defined by $\Pi(A_1,\ldots, \\ A_k)=A_1 A_2 \ldots A_k$ is a completely positive multilinear map of type 1. It can be claimed that the definition of $n$-positivity of type 1 is the minimal condition that provides the complete positivity for multilinear maps such as $\Pi$ on general $C^*$-algebras.
A version of the Stinespring theorem for completely contractive multilinear maps of type 1 was obtained by Paulsen and Smith \cite{smith}. They also gave a correspondence between completely bounded multilinear maps of type 1 $\Phi : \bigoplus_{i=1}^k \mathcal{S}_i \to \mathscr{B} $ and completely bounded linear maps $\Phi : \bigotimes_{i=1}^k \mathcal{S}_i \to \mathscr{B} $
on the tensor products endowed with the Haagerup norm; see \cite[Propositions 1.3 and 3.1]{smith}. Some other important studies on Stinespring's theorem on Hilbert $C^*$-modules and $C^*$-algebras in the setting of multilinear maps, concerning the definitions of complete positivity of type 1 and complete boundedness of type 1, can be found in \cite{Bhat2, heo1, heo2,  joita1}. Furthermore, some covariant representations for a symmetry semigroup $C^*$-dynamical system were presented in \cite{bel1, bel2}.\\
\textbf{Type 2.} In accordance with the usage of linear cases, a multimap $\Phi : \bigoplus_{i=1}^k \mathcal{S}_i \to \mathscr{B} $ is called $n$-positive of type 2 if the induced map $\Phi_n : \bigoplus_{i=1}^k M_n \left(\mathcal{S}_i\right) \to M_n\left( \mathscr{B}\right) $ given by $ \Phi_n\left(\left[A_{ij}^1\right],\ldots, \left[A_{ij}^k\right]\right)= \left[ \Phi\left(A_{ij}^1,\ldots, A_{ij}^k\right)\right]_{i,j=1}^n$ is positive. We say that $\Phi$ is completely positive of type 2 if $\Phi_n$ is positive for all $n\in\mathbb{N}$. In the framework of positive maps between $C^*$-algebras, this definition was first presented in the work of Ando and Choi \cite{CHOICOM}. They showed that every completely positive map of type 2 is uniquely decomposed to completely positive mixed homogeneous maps of type 2. It allows us to reduce the study of completely positive maps of type 2 to completely positive mixed homogeneous maps and completely positive multilinear maps. Moreover, they gave a Stinespring type representation for linear completely positive maps of type 2 between $C^*$-algebras; see \cite[Theorems 5, 6, and 7]{CHOICOM}. The authors of \cite{HIAI} showed that every completely positive multilinear map $\varphi :\bigoplus_{i=1}^k \mathscr{A}_i \to \mathscr{B}$ of type 2 between $C^*$-algebras corresponds to a completely positive linear map $\rho :\bigotimes_{i=1}^k \mathscr{A}_i \to \mathscr{B}$ with respect to the projective tensor product  (i.e., $C^*$-tensor product with respect to the largest $C^*$-crossnorm). Dong \cite{shon} considered the definition of $n$-positivity of type 2 and generalized some properties of $n$-positive linear maps for $n$-positive and completely positive multilinear maps. Recently, Dehghani, the second author, and Seo \cite{DKS} proved a Choi--Davis--Jensen type inequality for positive multilinear maps and submultiplicative operator convex functions. \\
We give some examples to show that the definitions of $n$-positivity of type 1 and of type 2 are generally independent of each other.
\begin{example}
Consider the map $\varphi : M_2(C[0,1])^{(4)} \to M_2(C[0,1])$ given by
\[\varphi (A_1, A_2, A_3, A_4) = A_1 \circ A_2 \circ A_3 \circ A_4,\]
where $C[0,1]$ is the unital $C^*$-algebra of all complex valued continuous functions on $[0,1]$ with the identity ${1}$.
Since the Hadamard product of positive matrices with entries in a commutative unital $C^*$-algebra is a positive matrix, it is easy to check that $\varphi$ is a completely positive multilinear map of type 2. However, we show that $\varphi$ is not even a positive map of type 1. Let
\[ A_1= \begin{bmatrix}
x & x \\ {1} & 0
\end{bmatrix}, \qquad A_2= \begin{bmatrix}
{1} & {1} \\ {1} & {1}
\end{bmatrix}, \qquad A_3= \begin{bmatrix}
{1} & {1} \\ {1} & {1}
\end{bmatrix}, \qquad A_4= \begin{bmatrix}
x & {1}\\ x & 0
\end{bmatrix}.
\]
Then, trivially, $(A_1, A_2, A_3, A_4)= (A_4^*, A_3^*, A_2^*, A_1^*)$. However,
$$\varphi (A_1, A_2, A_3, A_4)= \begin{bmatrix}
x^2 & x \\ x & 0
\end{bmatrix} \not \geq 0.$$
\end{example}

 As mentioned earlier, if $\mathscr{A}$ is a $C^*$-algebra, then the map $\Pi : \mathscr{A}^{(k)} \to \mathscr{A}$ defined by $\Pi(A_1, \ldots, A_k)=A_1 A_2 \ldots A_k$ is a completely positive multilinear map of type 1. However, if $\mathscr{A}$ is a noncommutative $C^*$-algebra and $k\geq 2$, then $\Pi$ is not a positive map of type 2.

Finally, we present an easy example to show that the $n$-positivity of type 1, at least in the setting of nonmultilinear maps, can be a strong condition on a multimap. Let $k$ be an even natural number. Consider the multimap $\Lambda: \mathscr{A}^{(k)} \to \mathscr{A}^{(k-1)}$ given by $\Lambda (A_1, \ldots, A_k)=(A_1, \ldots, A_{k-1})$. Then $\Lambda$ is a completely positive linear map of type 2. However, $\Lambda$ is not even a positive map of type 1. Indeed, we have $(-I, I, \ldots, I, -I)= (-I^*, I^*, \ldots, I^*, -I^*)$, but
 $\Lambda(-I, I, \ldots, I, -I)=(-I, I, \ldots, I)\not \geq 0$.

 Accordingly, since the main part of our purpose is to study of (not necessarily linear or multilinear) positive maps, we consider the definition of $n$-positivity of type 2, and for simplicity, we use the terms ``$n$-positive'' and ``completely positive'' instead of ``$n$-positive of type 2'' and ``completely positive of type 2", respectively.

 The main goal of this paper is to investigate (not necessarily linear or multilinear) positive maps between $C^*$-algebras and to explore applications of positive maps in quantum information theory.

The preliminary Section 2 is devoted to some examples and basic properties of positive maps. The relationships between completely positive maps, completely positive multilinear maps, and completely positive $(m, n)$-mixed homogeneous maps are reviewed.

In Section 3, we prove that any positive multilinear map between $C^*$-algebras is separately weak$^*$-continuous, while the joint weak$^*$-continuity is equivalent to a condition on the domain $C^*$-algebras, that is the joint continuity of the multiplications. The separately weak$^*$-continuous extension of multilinear maps enables us to present a decomposition for tracial multilinear between general $C^*$-algebras. Moreover, we demonstrate decompositions of tracial completely positive maps between $C^*$-algebras.

In Section 4, we develop some classical uncertainty relations by considering some classes of positive maps between $C^*$-algebras. Among other results, some uncertainty relation inequalities for commuting observables in a composite physical system are given: if $ \Phi: \mathscr{A}\bigotimes \mathscr{B} \to \mathscr{C}$ is a unital positive linear map between unital $C^*$-algebras and $A\otimes I_\mathscr{B}$ and $ I_ \mathscr{A} \otimes B$ are compatible observables in the composite system, then
 \[ {\rm Var}_\Phi(A\otimes I_\mathscr{B}) {\rm Var}_\Phi(I_\mathscr{A}\otimes B) \geq \frac{1}{16} \frac{ \left|\Phi([A, C]) \Phi([B,D])\right|^2}{(\|C-\alpha I_\mathscr{A}\| \|D-\beta I_\mathscr{B}\|)^2 }\]
 for all self-adjoint elements $ C\in \mathscr{A}$ and $D\in\mathscr{B}$ and every $\alpha, \beta \in \mathbb{C}$ with nonzero $ \|C-\alpha I_\mathscr{A}\|$ and $\|D-\beta I_\mathscr{B}\|$.

\section{Basic Properties and Examples}
In this section, we present some examples of positive maps between $C^*$-algebras. Moreover, we express the relationship between completely positive and completely positive multilinear maps.

It is known from \cite[Theorem 5.1]{Hiai} that if $\alpha\geq n-2$, then the power function $\phi_\alpha:\mathbb{R}\to \mathbb{R}$ given by $\phi_\alpha(x)= |x|^\alpha$ is an $n$-positive map. In the other word, if $[x_{ij}]_{i,j=1}^n$ is a positive matrix and $\alpha\geq n-2$, then the matrix $[|x_{ij}|^\alpha]_{i,j=1}^n$ is positive. In addition, according to \cite[Theorems 2.2 and 2.4 ]{Fitz}, if $\alpha $ is not an integer, then the condition $\alpha \geq n-2$ for $\phi_\alpha$ is sharp. Moreover, if $[x_{ij}]_{i,j=1}^n \in M_n(\mathbb{C})$ is a positive matrix, then $[\bar{x}_{ij}]_{i,j=1}^n \in M_n(\mathbb{C})^+$. Thanks to the Schur product theorem, we have $[|{x}_{ij}|^2]_{i,j=1}^n \in M_n(\mathbb{C})^+$. Therefore, if $\alpha\geq n-2$, then the map $\varphi_\alpha:\mathbb{C}\to \mathbb{C}$ defined by $\varphi_\alpha(x)= |x|^{2\alpha}$ is an $n$-positive map. We are ready to give the following examples.
\begin{example}\label{exam1}
 If $\alpha_1, \ldots, \alpha_k $ are not integers, then for every $1\leq i\leq k$ consider the maps $\phi^{\alpha_i}: M_m(\mathbb{R}) \to M_m (\mathbb{R})$ given by $\phi^{\alpha_i}\left(\left[x_{ij}\right]_{i,j=1}^m\right)= \Big[\left|x_{i,j}\right|^{\alpha_i}\Big]_{i,j=1}^m$. According to what we said above, for every $1\leq i\leq k$, the multimap $\phi^{\alpha_i}$ is $n$-positive if and only if $\alpha_i \geq nm-2$. Now define the map $\Phi: M_m(\mathbb{R})^{(k)} \to M_m (\mathbb{R})$ by \[\Phi \left(\left[x_{ij}^1\right]_{i,j=1}^m, \ldots, \left[x_{ij}^k\right]_{i,j=1}^m\right)= \phi^{\alpha_1}\left( \left[x_{i,j}^1\right]_{i,j=1}^m\right) \circ \cdots \circ \phi^{\alpha_k} \left(\left[x_{i,j}^k\right]_{i,j=1}^m\right).\] One observes that $\Phi$ is $n$-positive if and only if ${\rm min} \{ \alpha_i : 1\leq i\leq k\} \geq nm-2$.
\end{example}
\begin{example}
Let $\alpha_1, \ldots, \alpha_k $ be some positive real numbers. Then for every $1\leq i\leq k$, consider the maps $\varphi^{\alpha_i}: M_m(\mathbb{C}) \to M_m (\mathbb{C})$ given by $\varphi^{\alpha_i}\left(\left[x_{ij}\right]_{i,j=1}^m\right)= \Big[\left|x_{i,j}\right|^{2\alpha_i}\Big]_{i,j=1}^m$. Due to the descriptions in the paragraph preceding Example \ref{exam1}, if $\alpha_i \geq nm-2$ ($1\leq i\leq k$), then the multimap $\varphi^{\alpha_i}$ is an $n$-positive map, $1\leq i\leq k$. Let us define the map $\Psi: M_m(\mathbb{C})^{(k)} \to M_m (\mathbb{C})$ by \[\Psi \left(\left[x_{ij}^1\right]_{i,j=1}^m, \ldots, \left[x_{ij}^k\right]_{i,j=1}^m\right)= \varphi^{\alpha_1}\left( \left[x_{i,j}^1\right]_{i,j=1}^m\right) \circ \cdots \circ \varphi^{\alpha_k} \left(\left[x_{i,j}^k\right]_{i,j=1}^m\right).\] If ${\rm min} \{ \alpha_i : 1\leq i\leq k\} \geq nm-2$, then $\Psi$ is an $n$-positive multimap. In particular, if $\alpha_1, \ldots, \alpha_k $ are natural numbers, then $\Psi$ is a completely positive $(\alpha_1\times \cdots \times \alpha_k, \alpha_1\times \cdots \times \alpha_k)$-mixed homogeneous map.
\end{example}

\begin{example}
Let the multilinear map $\Phi:{M}_{m_1}(\mathbb{C})\times \cdots \times {M}_{m_k}(\mathbb{C})\to {M}_{m_1 \times \cdots \times m_k}(\mathbb{C})$ be defined by $\Phi(A_1, \ldots, A_k)=A_1\otimes \cdots \otimes A_k$. Then $\Phi$ is a completely positive multilinear map. As another example, it can be seen that the multilinear map $\Psi:{M}_m^{(k)}(\mathbb{C})\to\mathbb{C}$ given by $\Psi(A_1, \ldots, A_k)=\mathrm{Tr} A_1 \cdots \mathrm{Tr} A_k$ is a tracial completely positive multilinear map.

The map $\Theta: M_2(\mathbb{C})^{(2)} \to M_2(\mathbb{C}) $ defined by $ \Theta(A,B)=A^T\otimes B^T$ is positive and multilinear. However, $\Theta$ is not $2$-positive. In fact, consider the positive matrix $E=\left[\begin{matrix}
 E_{11} & E_{12} \\
 E_{21} & E_{22}
 \end{matrix}\right]$,
 where $E_{ij}$'s are the matrix units in $M_2(\mathbb{C})$. However, $\Theta_2(E,E)$ is not positive; we have
 \begin{align*}
 \Theta_2(E,E)=\left[\begin{matrix}
 \Theta(E_{11},E_{11}) & \Theta(E_{12},E_{12}) \\ \Theta(E_{21},E_{21}) & \Theta(E_{22},E_{22}) \end{matrix}\right]=\left[\begin{matrix}
 E_{11}\otimes E_{11} & E_{21}\otimes E_{21}\\ E_{12}\otimes E_{12} & E_{22}\otimes E_{22}
 \end{matrix}\right]\ngeq0.
 \end{align*}
 \end{example}

\subsection{Some Basic Properties}
It is known from \cite[Proposition 12]{shon} that if $\mathscr{A}_i$, $1\leq i\leq k$, are commutative unital $C^*$-algebras and $\mathscr{B}$ is a unital $C^*$-algebra, then every unital positive multilinear map $\Phi : \bigoplus_{i=1}^k\mathscr{A}_i \to \mathscr{B} $ is completely positive. Hence,
if $\textbf{A}=(A_1,\ldots, A_k)$ is a normal element, then the restriction of $\Phi$ on $C^*(\textbf{A}, \textbf{I})$, the $C^*$-algebra generated by $\textbf{A}$ and $\textbf{I}$, is completely positive.
Therefore, the positivity of the matrix
\[ \begin{bmatrix}
\textbf{A}^* \textbf{A} & \textbf{A}^*\\
\textbf{A} & \textbf{I}
\end{bmatrix}\in \bigoplus_{i=1}^k M_2(C^*(\textbf{A}, \textbf{I})) \]
 and the complete positivity of $\Phi$ on $C^*(\textbf{A}, \textbf{I})$ imply that
 \[ \begin{bmatrix}
\Phi(\textbf{A}^* \textbf{A}) & \Phi( \textbf{A}^*)\\
\Phi( \textbf{A}) & \Phi(\textbf{I})
\end{bmatrix}\geq 0.\]
 Hence, by applying Lemma \ref{epsilpo}, we arrive at the following Choi type inequality:
 \begin{align}\label{kadison}
 \Phi(\textbf{A}^*\textbf{A})\geq \Phi(\textbf{A}^*)\Phi(\textbf{A}).
 \end{align}
 In particular, if $\textbf{A}$ is a self-adjoint operator, then
 \begin{align}\label{kadison2}
 \Phi(\textbf{A}^2)\geq \Phi(\textbf{A})^2.
 \end{align}
Furthermore, if $\Phi : \bigoplus_{i=1}^k \mathscr{A}_i \to \mathscr{B} $ is a $2$-positive map, then inequality \eqref{kadison} is also valid for arbitrary elements in $\bigoplus_{i=1}^k \mathscr{A}_i$.\\
A positive multilinear map is clearly self-adjoint (see \cite[Lemma 1]{shon}), but the converse is not true in general. Moreover, every positive multilinear map is bounded; see \cite{Bh,DM3}.
However, if $\bigoplus_{i=1}^k \mathscr{A}_i$ is unital  $\Phi : \bigoplus_{i=1}^k \mathscr{A}_i \to \mathscr{B} $ is $2$-positive, then it is self-adjoint. Indeed, it is easy to check that for every $\textbf{A}=(A_1,\ldots, A_k)\in \bigoplus_{i=1}^k \mathscr{A}_i$, the matrix
\[ \begin{bmatrix}
\|\textbf{A}\| \textbf{I} &  \textbf{A}^*\\
 \textbf{A} & \|\textbf{A}\| \textbf{I}
\end{bmatrix}\in \bigoplus_{i=1}^k M_2(\mathscr{A}_i) \]
is positive. Hence, $2$-positivity of $\Phi$ ensures that
\[ \begin{bmatrix}
\Phi\left(\|\textbf{A}\| \textbf{I}\right) & \Phi\left( \textbf{A}^*\right)\\
\Phi\left( \textbf{A}\right) & \Phi\left(\|\textbf{A}\| \textbf{I}\right)
\end{bmatrix}\geq 0, \]
 which immediately implies $\Phi\left( \textbf{A}^*\right) = \Phi\left( \textbf{A}\right)^*$. \\
 The monotonicity of positive multilinear maps on positive elements can be easily proved. However, for positive multimaps, the positivity is not enough to imply the monotonicity. For example, the map $\Phi : M_2(\mathbb{C})^{(k)} \to \mathbb{C}$ given by
 \[\Phi([x_{ij}^1],\ldots, [x_{ij}^k] )=\begin{cases}\| ([x_{ij}^1],\ldots, [x_{ij}^k] )\| & \text{if\ } ([x_{ij}^1],\ldots, [x_{ij}^k] )\leq \textbf{I}, \\ 0 & \text{otherwise}, \end{cases}\] is a positive multimap, which is not monotone. However, if $\Phi :\mathscr{A} \to \mathscr{B}$ is a $2$-positive multimap between $C^*$-algebras, then it is monotone on positive elements. Indeed, it follows from \cite[Inequality 4]{DM3} that
 \[ A\geq B\geq 0 \Longleftrightarrow \begin{bmatrix}
 A & B \\ B & B
 \end{bmatrix}\geq 0.\]
 Hence, the $2$-positivity of $\Phi$ immediately shows the assertion to be held. More generally, it is also ensures that every $2n$-positive map $\Phi: \mathscr{A} \to \mathscr{B}$ between $C^*$-algebras is $n$-monotone (i.e., $\Phi$ is monotone on $M_n(\mathscr{A})^+$); see \cite{DM3, DMK3, Nagisa} for more details.

\subsection{An Overview to Nonlinear Completely Positive Maps}\label{overviewnonlinear}
An interesting and important result on completely positive maps between $C^*$-algebras is a result of Ando and Choi \cite[Theorem 2]{CHOICOM} as follows. Let $\Phi:\mathscr{A} \to \mathscr{B}$ be a completely positive map between $C^*$-algebras.
Then there are completely positive $(m,n)$-mixed homogeneous maps
 $\Phi_{m,n} : \mathscr{A}\to \mathscr{B}$ ($m,n= 0, 1, \ldots $) such that
\begin{equation}\label{namayeshcom}
\Phi (A)= \sum_{m=0}^\infty \sum_{n=0}^\infty \Phi_{m,n}(A) \qquad (A\in \mathscr{A}),
\end{equation}
where the convergence of the series is in the operator norm topology on $\mathscr{B}$. Also, the structure of $\Phi_{m,n}$'s reads as follows.
For every $A\in \mathscr{A}$, the map $F : \mathbb{C} \to \mathscr{B}$ defined by $F(z)=\Phi(zA)$ is infinitely
differentiable and
\begin{equation}
\Phi_{m,n}(A)=\frac{1}{m! n!} \frac{\partial^{m+n}}{\partial^m z \partial^n \bar{z}} \Phi(zA)|_{z=0} \qquad (z\in \mathbb{C}).
\end{equation}
The structure of $\Phi_{m,n}$'s ensures that numerous properties of $\Phi$ can be established for $\Phi_{m,n}$'s as well. For instance, if $\Phi$ is tracial, so are $\Phi_{m,n}$'s ($m,n= 0, 1, \ldots $). Moreover, $\Phi_{0,0}(A)=\Phi(0)$.\\
As mentioned in \cite[Theorem 1.4]{HIAI} and \cite[Theorem 5]{CHOICOM}, if $\Phi_{m.n}:\mathscr{A} \to \mathscr{B}$ is a completely positive $(m,n)$-mixed homogeneous map with $m+n>0$, then there is a completely positive $(m+n)$-linear map ${\Phi}^{(m,n)}: \mathscr{A}^{(m)}\bigoplus \bar{\mathscr{A}}^{(n)} \to \mathscr{B}$ such that
\begin{equation}\label{homomulti}
\Phi_{m,n}(A)= \Phi^{(m,n)}(A^{(m)}, \bar{A}^{(n)}),
\end{equation}
where $\bar{\mathscr{A}}$ is the $C^{*}$-algebra conjugate to $\mathscr{A}$. Recall from \cite[Section 1]{HIAI} that the $C^{*}$-algebra $\bar {\mathscr{A}}$ conjugated to $\mathscr{A}$ is defined by the
same underlying set as $\mathscr{A}$ and as the same as addition, multiplication, involution,
and norm of $\mathscr{A}$. If we denote by $\bar{X}$ the
element in $\bar{\mathscr{A}}$ corresponding to $X\in \mathscr{A}$, then the scalar multiplication from $ \mathbb{C} \times \bar{\mathscr{A}}$ to $\bar{\mathscr{A}}$ is defined by $(z, \bar{A})\mapsto \overline{\bar{z}A}$, where $\bar{z}$ means the complex conjugate of $z$. The map $A \mapsto\bar{A}$ from $\mathscr{A}$ to $\bar {\mathscr{A}}$ is a conjugate linear
isomorphism and therefore is a completely positive map.
The above discussion together with the $(m,n)$-mixed homogeneous representation \eqref{namayeshcom} for completely positive maps between $C^*$-algebras show that the study of completely positive maps can be reduced to that of completely positive multilinear maps. This is a motivation for our investigation.

\section{Decomposition of Tracial Positive Maps }
 Let $X_1, \ldots, X_k, Y$ be Banach spaces and let $\Gamma: X_1\times X_2\times \cdots \times X_k \to Y$ be a bounded multilinear map. Then there exists a multilinear extension of $\Gamma$ to a map $\tilde{\Gamma}: X_1^{**}\times X_2^{**}\times \cdots \times X_k^{**} \to Y^{**}$, which is called the Aron--Berner extension of $\Gamma$. In the case when $X_1, \ldots, X_k$ are $C^*$-algebras and $Y$ is dual space of a complex Banach space, Johnson, Kadison, and Ringrose \cite{johnson} showed that the Aron--Berner extension of a bounded separately ultraweak-weak$^*$-continuous multilinear map $\Gamma$ is unique, separately ultraweak-weak$^*$-continuous and preserves the norm of $\Gamma$. We refer the interested reader to \cite{bombal, maitl, Pera} for more information. Although the goal of this paper is not to study the mentioned subject, however, they play key roles in obtaining some results of our special interest. In particular, Theorem \ref{th-e3} shows that, for acquiring the joint weak$^*$-continuity instead of separate weak$^*$-continuity, a strong condition on the multiplication in the domain $C^*$-algebras would be needed.
Recall that, according to the Goldstein theorem, every Banach space $X$ can be considered as a dense subspace of $X^{**}$ with respect to the weak$^*$-topology, that is, $\sigma(X^{**}, X^*)$.
\begin{lemma}\label{arbexte}
 Let $\mathscr{A}_1, \cdots, \mathscr{A}_k $ be unital $C^*$-algebras and $\mathscr{B}$ be a $C^*$-algebra. Then every positive multilinear map $\Phi:\bigoplus_{i=1}^k \mathscr{A}_i \to \mathscr{B}$ is separately weak$^*$-continuous. That is, for every net $\{ A_{\alpha}\} \subseteq \mathscr{A}_j$ $(1\leq j\leq k)$ and all $C_i \in \mathscr{A}_i$ $(i\neq j,\ 1\leq i\leq k$), it holds that
\[ A_{\alpha} \xrightarrow{{\rm weak}^*} A\quad \Longrightarrow \quad \Phi(C_1, \ldots, A_{\alpha}, \ldots, C_k) \xrightarrow{{\rm weak}^*} \Phi(C_1, \ldots, A, \ldots, C_k). \]
\end{lemma}
\begin{proof}
 Since positive linear maps are bounded, the conclusion is valid in the case when $k=1$. This immediately implies that every positive multilinear map $\Psi : \bigoplus_{i=1}^k \mathscr{A}_i \to \mathscr{B}$ is separately weak$^*$-continuous.
For the sake of convenience, we shortly illustrate it as follows. Fix $j\in \{1, 2,\ldots, k\}$. For every $1\leq i\leq k$ with $ i\neq j$, let $C_i\in\mathscr{A}_i$. Let the map $\Phi_{(j)} : \mathscr{A}_j \to \mathscr{B}$ be defined by \[\Phi_{(j)}(X) = \Phi(C_1, \ldots, C_{j-1}, X, C_{j+1}, \ldots, C_k).\] Due to the multilinearity of $\Phi$, without loss of generality, we may assume that $C_i \in \mathscr{A}_{i}^+$ ($ 1\leq i\leq k, i\neq j$). Accordingly, $\Phi_{(j)}$ would be a positive linear map, which is certainly bounded. Consequently, for every $f\in \mathscr{B}^*$, it holds that $f\circ \Phi_{(j)} \in\mathscr{A}_j^*$. Therefore, the assumption $ A_{\alpha} \xrightarrow{{\rm weak}^*} A$ implies that
\small{
\begin{equation*}
f\circ \Phi\left(C_1, \ldots,A_{\alpha}, \ldots, C_k\right)=f\circ \Phi_{(j)}\left(A_{\alpha}\right) \to f\circ \Phi_{(j)}\left(A\right)=f\circ \Phi\left(C_1, \ldots, A, \ldots, C_k\right)
\end{equation*}}
 for every $f\in \mathscr{B}^*$, as we claimed. \\
\end{proof}

A general extension theorem for multilinear operators was presented by Bombal et al. \cite[Theorem 1]{bombal} as follows: Let $X_1, \ldots, X_k, Y$ be Banach spaces, and let $\Gamma: X_1\times X_2\times \cdots \times X_k \to Y$ be a bounded multilinear map such that for all $i \neq j$; $1\leq i, j\leq k$, every
bounded linear map $T: X_i \to X_j^*$ is weakly compact (i.e., $T$ maps the unit sphere of $X_i$ into a weakly relatively compact set in $X_j^*$).
Then there is a unique bounded multilinear map $\tilde{\Gamma}: X_1^{**}\times X_2^{**}\times \cdots \times X_k^{**} \to Y^{**}$ which extends ${\Gamma}$ and is separately weak$^*$-continuous. Moreover, $\|\Gamma\|= \|\tilde{\Gamma}\|$. Furthermore, according to \cite[Corollary 2]{bombal} if either $\Gamma$ is a weakly compact operator or every bounded linear map from $X_i$ into $Y$ is weakly compact, then $\tilde{\Gamma}$ is $Y$-valued. The existence part of the next result is known (see \cite[Theorem 3.2]{johnson} and \cite[Lemma 3.1]{maitl}). We prove it by employing the following lemma due to Argerami \cite{Argerami} for the completeness of discussion.

\begin{lemma}\label{ARGE}
 $\mathscr{A}^+$ is a weak$^*$-dense subset of ${\mathscr{A}^{**}}^+$ for any $C^*$-algebra $\mathscr{A}$.
\end{lemma}
\begin{proof}
It is enough to show that $S_\mathscr{A}^+$ is weak$^*$-dense in $S_\mathscr{A^{**}}^+$ where $S_\mathscr{X}^+= \{ X \in \mathscr{X} : \|X\|=1 , X\geq 0\}$. Let $\mathscr{A}^{''}$ be the enveloping von Neumann algebra of $\mathscr{A}$. By the Sherman--Takeda theorem \cite[Chapter III, Theorem 2.4]{TAK1}, $\mathscr{A}^{**}$ with respect to the weak$^*$-topology can be identified with $\mathscr{A}^{''}$ with respect to $\sigma$-weak topology and the identification preserves the positivity. According to \cite[Theorem 1.7.3]{Davidson}, $S_\mathscr{A}^+$ is strong operator-dense in $S_\mathscr{A^{''}}^+$, which ensures $S_\mathscr{A}^+$ is weak operator-dense in $S_\mathscr{A^{''}}^+$. The weak operator topology and $\sigma$-weak topology coincide on the unit ball, so $S_\mathscr{A}^+$ is $\sigma$-weak-dense in $S_\mathscr{A^{''}}^+$. Therefore, the mentioned identification implies that $S_\mathscr{A}^+$ is weak$^*$-dense in $S_\mathscr{A^{**}}^+$.
\end{proof}

\begin{lemma}\label{extension1}
Every positive multilinear map $ \Phi: \bigoplus_{i=1}^k \mathscr{A}_i \to \mathscr{B}$ between unital $C^*$-algebras has a unique, separately weak$^*$-continuous extension $\tilde{\Phi}: \bigoplus_{i=1}^k \mathscr{A}_i^{**} \to \mathscr{B}^{**}$, such that $\|\Phi\|=\|\tilde{\Phi}\|$ and $\tilde{\Phi}$ is positive and multilinear. Moreover, if either $\mathscr{B}$ is a von Neumann algebra or $\Phi$ is weakly compact, then $\tilde{\Phi}$ is $\mathscr{B}$-valued.
\end{lemma}
\begin{proof}
Due to \cite[Corollary II.9]{Akemann}, every bounded linear map from a $C^*$-algebra to the predual of a von Neumann algebra is weakly compact. Hence, every bounded linear map form $\mathscr{A}_i$ to $\mathscr{A}_j^*$ is weakly compact for $1\leq i,j\leq k$. Now the conditions of \cite[Theorem 1]{bombal} (mentioned in the above paragraph) are fulfilled for the multilinear map $\Phi: \bigoplus_{i=1}^k \mathscr{A}_i \to \mathscr{B}$. Therefore, the existence of a unique separately weak$^*$-continuous extension $\tilde{\Phi}: \bigoplus_{i=1}^k \mathscr{A}_i^{**} \to \mathscr{B}^{**}$ with $\|\tilde{\Phi}\|=\|\Phi\|$ is guaranteed.

 The Goldstine theorem follows that if $\mathscr{X}$ is a Banach space, then $\mathscr{X}$ is a weak$^*$-dense subset of $\mathscr{X}^{**}$. This ensures that for every $j=1,\dots,k$ and every $A_j \in \mathscr{A}_j^{**} $, there exists a net $A_{\alpha^{(j)}} $ in $ \mathscr{A}_j$ that is weak$^*$-convergent to $A_j$.
By the construction \cite[Theorem 1]{bombal}, we can define the map $\tilde{\Phi}$ from $\bigoplus_{i=1}^k\mathscr{A}_i^{**}$ to $\mathscr{B}^{**}$ by the iterated limits as follows:
\[(A_1, \ldots, A_k)\mapsto {\rm weak}^*-\lim_{\alpha^{(1)}}\cdots {\rm weak}^*-\lim_{\alpha^{(k)}}\Phi(A_{\alpha^{(1)}},\ldots, A_{\alpha^{(k)}})\]
that is the separately weak$^*$-continuous multilinear extension of $\Phi$. It follows from Lemma \ref{ARGE} that  $\tilde{\Phi}$ is positive.  In the case when $\Phi$ is weakly compact or $\mathscr{B}$ is a von Neumann algebra, the assertion follows from \cite[Corollary 2]{bombal} by using \cite[Corollary II.9]{Akemann}.
 \end{proof}

The next theorem shows that the joint weak$^*$-continuity of positive multilinear maps requires a strong condition on the domain $C^*$-algebras.
\begin{theorem}\label{th-e3}
 Let $\Omega$ be a set of unital $C^*$-algebras and $\mathscr{B}$ be a $C^*$-algebra. Then the followings are equivalent:
 \begin{enumerate}
 \item Every positive multilinear map $\Phi:\bigoplus_{i=1}^k \mathscr{A}_i \to \mathscr{B}$ is jointly weak$^*$-continuous in which $\mathscr{A}_1, \cdots, \mathscr{A}_k \in \Omega$.
 \item The multiplication is a jointly weak$^*$-continuous operation in every $\mathscr{A}\in \Omega$.
 \item If $\{A_{\alpha}\}_\alpha $ is a net in $\mathscr{A}$ for some $\mathscr{A}\in \Omega$, then
 \[ A_{\alpha} \xrightarrow{{\rm weak}^*} A \qquad \Longrightarrow \qquad |A_{\alpha}|^2 \xrightarrow{{\rm weak}^*} |A|^2.\]
 \end{enumerate}
\end{theorem}
\begin{proof}
 We show that (1) $\Longrightarrow$ (2)$\Longrightarrow$ (3) $\Longrightarrow$ (1).

 (1) $\Longrightarrow$ (2). Let $\mathscr{A} \in \Omega$ and $\{A_{\alpha}\}_\alpha, \{B_{\beta}\}_\beta \subset \mathscr{A}$ be some nets such that $A_{\alpha} \xrightarrow{{\rm weak}^*} A$ and $B_{\beta} \xrightarrow{{\rm weak}^*} B$. Define the map $\Phi : \mathscr{A} \bigoplus \mathscr{A} \to \mathscr{A} \bigotimes \mathscr{A}$ by
$\Phi(A,B) = A\otimes B$, where $\bigotimes$ denotes the spatial tensor product of $C^*$-algebras. Then it is easy to check that $\Phi$ is a completely positive multilinear map that is jointly weak$^*$-continuous if (1) is valid. Let $f\in \mathscr{A}^*$ be an arbitrary bounded linear functional. Then $f$ induces a bounded bilinear map $g: \mathscr{A} \bigoplus \mathscr{A} \to \mathbb{C}$ given by $g \left( A, B\right)= f(A B)$. Hence, there is a unique bounded linear functional $\tilde{g} \in\left(\mathscr{A} \bigotimes \mathscr{A}\right)^*$ such that $\tilde{g}(A\otimes B)= g(A,B)=f(AB)$. Therefore, by using assumption (1), we get
\begin{align*}
f(A_\alpha B_\beta)= \tilde{g}(A_\alpha \otimes B_\beta)= (\tilde{g}\circ \Phi)(A_\alpha, B_\beta) \xrightarrow{{\rm jointly}} (\tilde{g}\circ \Phi)(A, B) =\tilde{g}(A\otimes B)=f(AB),
\end{align*}
which proves (2).

(2) $\Longrightarrow$ (3). It is immediately follows from (2) according to the fact that the involution is
weak$^*$-continuous.

(3) $\Longrightarrow$ (1). We proceed by induction on $k$. Due to Lemma \ref{arbexte}, the conclusion is valid for $k=1$.
Let the conclusion be valid for $k=n$, let $ \Phi: \bigoplus_{i=1}^{n+1} \mathscr{A}_i \to \mathscr{B}$ be a positive multilinear map, and let $f$ be an arbitrary element of $ \mathscr{B}^* $. We show that if a net $\left\{ \left(A_{\alpha^{(1)}}, \ldots, A_{\alpha^{(n+1)}}\right)\right\}$ in $\bigoplus_{i=1}^{n+1} \mathscr{A}_i $
is weak$^*$-convergent to an element $(A_1,\ldots, A_{n+1})$, then
$$(f\circ\Phi)( A_{\alpha^{(1)}}, \ldots, A_{\alpha^{(n+1)}}) \rightarrow (f\circ\Phi) (A_1,\ldots, A_{n+1}).$$
According to the Jordan decomposition for bounded linear functionals on $C^*$-algebras, $f$ can be considered as a linear combination of four positive linear functionals on $\mathscr{B}$.
Hence, without loss of generality, we may assume that $f$ is a positive linear functional. Furthermore, since the addition, involution and scaler multiplication are weak$^*$-continuous and $\Phi$ is multilinear, we may assume that $A_i$'s as well as the members of $\left\{A_{\alpha^{(i)}}\right\}$'s ($ 1\leq i\leq n+1$) are self-adjoint.\\ We first state two facts as follows: \\
\textbf{Fact 1}. Let $\Psi: \bigoplus_{i=1}^{m+1} \mathscr{A}_i \to \mathscr{B}$ be a positive multilinear map and let $X \in \mathscr{A}_{{m+1}}$ be a positive element. Then $\Psi$ induces the positive $m$-linear map $\Psi_X: \bigoplus_{i=1}^{m} \mathscr{A}_i \to \mathscr{B}$ defined by $\Psi_X (C_1, \ldots, C_m)= \Psi( C_1, \ldots, C_m, X)$. Moreover, if the conclusion holds for $k=m$, then $\Psi_X$ is jointly weak$^*$-continuous.\\
\textbf{Fact 2}. We claim that if
\begin{itemize}
\item[(i)] $Z_i\in {\mathscr{A}_i}$ for $ 1\leq i \leq n$ are self-adjoint,
\item[(ii)] $\{Z_{\alpha^{(i)}}\}$ are some nets of self-adjoint elements in ${\mathscr{A}_i}$ $(1\leq i \leq n)$ with $Z_{\alpha^{(i)}}\xrightarrow{{\rm weak}^*}Z_i$,
\item[(iii)] $\{Z_{\alpha^{(n+1)}}\}$ is a net of self-adjoint elements in $\mathscr{A}_{n+1}$ that is weak$^*$-convergent to zero,
\end{itemize}
then $ \Phi(Z_{\alpha^{(1)}}, \ldots, Z_{\alpha^{(n+1)}}) $ is jointly weak$^*$-convergent to zero. To prove our claim, assume that the nets $ \{Z_{\alpha^{(1)}}\}, \ldots, \{Z_{\alpha^{(n)}}\}$ and $\{Z_{\alpha^{(n+1)}}\}$ enjoy properties ${\rm (i)}$--${\rm (iii)}$. Note that For every $ \left(\alpha^{(1)}, \ldots, \alpha^{(n+1)}\right)$, the elements $\left(I_1, \ldots, I_n, Z_{\alpha^{(n+1)}}\right)$ and \\ $\left(Z_{\alpha^{(1)}}, \ldots, Z_{\alpha^{(n)}}, I_{n+1} \right)$ are commutative and self-adjoint elements of $\bigoplus_{i=1}^{n+1} \mathscr{A}_i$, so the restriction of $\Phi$ on the $C^*$-subalgebra generated by the set \[\left\{\left(I_1, \ldots, I_n, Z_{\alpha^{(n+1)}}\right),\left(Z_{\alpha^{(1)}}, \ldots, Z_{\alpha^{(n)}}, I_{n+1} \right), \left(I_1, \ldots, I_{n+1}\right)\right\},\] as a commutative unital $C^*$-algebra,
is a completely positive multilinear map; see \cite[Proposition 12]{shon}. Hence,
\[ \Phi_2 \left(\begin{bmatrix}\left(I_1, \ldots, I_n, Z_{\alpha^{(n+1)}}\right) &0 \\ \left(Z_{\alpha^{(1)}}, \ldots, Z_{\alpha^{(n)}}, I_{n+1} \right) & 0
\end{bmatrix}\begin{bmatrix}\left(I_1, \ldots, I_n, Z_{\alpha^{(n+1)}}\right) &0 \\ \left(Z_{\alpha^{(1)}}, \ldots, Z_{\alpha^{(n)}}, I_{n+1} \right) & 0
\end{bmatrix}^* \right) \geq 0.
\]
Hence, by a computation, we get
\[ \begin{bmatrix} \Phi \left(I_1, \ldots, I_n, Z_{\alpha^{(n+1)}}^2\right) &\Phi \left(Z_{\alpha^{(1)}}, \ldots, Z_{\alpha^{(n)}}, Z_{\alpha^{(n+1)}}\right) \\ \Phi \left(Z_{\alpha^{(1)}}, \ldots, Z_{\alpha^{(n)}}, Z_{\alpha^{(n+1)}}\right)& \Phi \left(Z_{\alpha^{(1)}}^2, \ldots, Z_{\alpha^{(n)}}^2, I_{n+1}\right)
\end{bmatrix} \geq 0.
\]
Now, since positive linear functionals are completely positive, we have
\[ \begin{bmatrix} (f\circ \Phi) \left(I_1, \ldots, I_n, Z_{\alpha^{(n+1)}}^2\right) & (f\circ\Phi)\left(Z_{\alpha^{(1)}}, \ldots, Z_{\alpha^{(n)}}, Z_{\alpha^{(n+1)}}\right) \\ (f\circ\Phi)\left(Z_{\alpha^{(1)}}, \ldots, Z_{\alpha^{(n)}}, Z_{\alpha^{(n+1)}}\right) & (f\circ \Phi)\left(Z_{\alpha^{(1)}}^2, \ldots, Z_{\alpha^{(n)}}^2, I_{n+1}\right)
\end{bmatrix} \geq 0
\] for every positive linear functional $f\in\mathscr{B}^*$.
Therefore,
\begin{align}\label{ki1}
 \left((f\circ \Phi)\left(I_1, \ldots, Z_{\alpha^{(n+1)}}^2\right)\right)& \left( ( f\circ \Phi) \left(Z_{\alpha^{(1)}}^2, \ldots, Z_{\alpha^{(n)}}^2, I_{n+1}\right)\right)\nonumber\\
 &\qquad\geq \left((f\circ\Phi) )\left(Z_{\alpha^{(1)}}, \ldots, Z_{\alpha^{(n+1)}}\right)\right)^2.
\end{align}
Using assumption (3) and according to \textbf{Fact 1}, we realize that\\ $(f\circ \Phi) \left(Z_{\alpha^{(1)}}^2, \ldots, Z_{\alpha^{(n)}}^2, I_{n+1}\right)= (f\circ\Phi_{I_{n+1}}) \left(Z_{\alpha^{(1)}}^2, \ldots, Z_{\alpha^{(n)}}^2\right)$ is a (jointly) convergent net in $\mathbb{C}$, so the set
\[\left\{(f\circ \Phi) \left(Z_{\alpha^{(1)}}^2, \ldots, Z_{\alpha^{(n)}}^2, I_{n+1}\right); \left(\alpha^{(1)}, \ldots, \alpha^{(n)}\right) \geq \left(\alpha^{(1)}_0, \ldots, \alpha^{(n)}_0\right) \right\}\]
 is a bounded set for some $ \left(\alpha^{(1)}_0, \ldots, \alpha^{(n)}_0\right)$. Furthermore, since $\{Z_{\alpha^{(n+1)}}\}\xrightarrow{{\rm weak}^*} 0$, assumption (3) together with the separate weak$^*$-continuity of $\Phi$ (Lemma \ref{arbexte}) imply that $(f\circ \Phi)\left(I_1, \ldots, Z_{\alpha^{(n+1)}}^2\right)$ tends to zero. Hence, \eqref{ki1} gives $$(f\circ\Phi)\left(Z_{\alpha^{(1)}}, \ldots, Z_{\alpha^{(n)}}, Z_{\alpha^{(n+1)}}\right)\xrightarrow{{\rm jointly}} 0.$$
 This proves our claim.
Now assume that $\left\{ (A_{\alpha^{(1)}}, \ldots, A_{\alpha^{(n+1)}})\right\}$ is an arbitrary net of self-adjoint elements in $ \bigoplus_{i=1}^{n+1}\mathscr{A}_i$, which is weak$^*$-convergent to $(A_1,\ldots, A_{n+1})$. If $f\in \mathscr{B}^* $ is a positive functional, then we have
{\small
\begin{align*}
 (f \circ \Phi)&(A_{\alpha^{(1)}}, \ldots, A_{\alpha^{(n+1)}}) - (f\circ \Phi)(A_1,\ldots, A_{n+1})\\ & = (f\circ \Phi) (A_{\alpha^{(1)}}, \ldots, A_{\alpha^{(n+1)}}) - (f\circ \Phi)(A_1,\ldots, A_{n+1})
 \\ & \ \ + (f\circ \Phi) (A_{\alpha^{(1)}}, \ldots, A_{\alpha^{(n)}}, A_{n+1} ) - (f\circ \Phi) (A_{\alpha^{(1)}}, \ldots, A_{\alpha^{(n)}}, A_{n+1})\\ & = (f\circ \Phi) \left(A_{\alpha^{(1)}}, \ldots, A_{\alpha^{n}}, [A_{\alpha^{(n+1)}} -A_{n+1}]\right) \\ & \ \ + (f\circ \Phi) (A_{\alpha^{(1)}}, \ldots, A_{\alpha^{(n)}}, A_{n+1} )- (f\circ \Phi)(A_1,\ldots, A_{n+1}) \\ & = (f\circ \Phi) (A_{\alpha^{(1)}}, \ldots, A_{\alpha^{(n)}}, [A_{\alpha^{(n+1)}} -A_{n+1}]) \\ &\ \ + \left( (f\circ \Phi_{ A_{n+1}}) (A_{\alpha^{(1)}}, \ldots, A_{\alpha^{(n)}} )- (f\circ \Phi_{ A_{n+1}})(A_1,\ldots, A_{n}) \right).
\end{align*}}
Since $[A_{\alpha^{(n+1)}} -A_{n+1}]$ is weak$^*$-convergent to 0, \textbf{Fact 2} implies that
$$ (f\circ \Phi) (A_{\alpha^{(1)}}, \ldots,A_{\alpha^{(n)}}, [A_{\alpha^{(n+1)}} -A_{n+1}])\xrightarrow{{\rm jointly}} 0.$$
 Moreover, it follows from \textbf{Fact 1} that
\[\left( (f\circ \Phi_{ A_{n+1}}) (A_{\alpha^{(1)}}, \ldots, A_{\alpha^{(n)}} )- (f\circ \Phi_{ A_{n+1}})(A_1,\ldots, A_{n}) \right)\xrightarrow{{\rm jointly}} 0.\]
 Consequently, $\Phi: \bigoplus_{i=1}^{n+1} \mathscr{A}_i \to \mathscr{B}$ is jointly weak$^*$-continuous, which means that the conclusion holds for $k=n+1$.
\end{proof}

We remark that in the case of the joint weak$^*$-continuity of positive multilinear maps, the extended mapping in Lemma \ref{extension1} is weak$^*$-continuous.

Choi and Tsui \cite{choi} showed that for every tracial positive linear map $\Phi: \mathscr{A} \to \mathbb{B}(\mathscr{H})$, there exist a compact Hausdorff $\Omega$, a tracial positive linear map $\varphi_1: \mathscr{A} \to C(\Omega)$, and a tracial positive linear map $\varphi_2: C(\Omega) \to \mathbb{B}(\mathscr{H})$ such that $\Phi=\varphi_2 \circ \varphi_1$. This result is very useful. At first glance, it ensures the complete positivity of tracial positive linear maps. Moreover, it can be used to improve the properties of tracial functionals to tracial positive linear maps between $C^*$-algebras.\\
We aim to give an extension of Choi--Tsui's Theorem in the setting of positive multilinear and completely positive maps. To achieve our results, we state some lemmas as follows.
\begin{lemma}\label{oneone}
Let $ \Phi: \mathscr{A} \to \mathscr{B}$ be a $3$-positive map between unital $C^*$-algebras. If $ P\geq Q \in \mathscr{A}^+$ such that $\Phi(P)=\Phi(Q)$, then $\Phi(P-Q)=\Phi(0)$.
\end{lemma}
\begin{proof}
If $P\geq Q\geq 0$, then we can write
\[ \begin{bmatrix}
 P-Q & P-Q & 0 \\ P-Q & P-Q & 0 \\ 0 & 0 & 0
\end{bmatrix} \geq 0, \qquad \begin{bmatrix}
0 &0 & 0 \\ 0 & Q & Q \\ 0 & Q & Q
\end{bmatrix} \geq 0.
\] Summing the above matrices and then utilizing $3$-positivity of $\Phi$ give
 \[ \begin{bmatrix}
 \Phi(P-Q) & \Phi(P-Q) & \Phi(0) \\ \Phi(P-Q) & \Phi(P) & \Phi(Q)\\ \Phi(0) & \Phi(Q) & \Phi(Q)
\end{bmatrix} \geq 0.
\]
If $\varepsilon >0$ is a real number, then it follows from the positivity of above matrix that
{\small
\[ \begin{bmatrix}
 \Phi(P-Q)+\varepsilon I & \Phi(P-Q)+ \varepsilon I & \Phi(0)+ \varepsilon I & 0 \\ \Phi(P-Q)+ \varepsilon I & \Phi(P)+ \varepsilon I & \Phi(Q)+ \varepsilon I& 0 \\ \Phi(0)+ \varepsilon I & \Phi(Q)+ \varepsilon I & \Phi(Q)+ \varepsilon I & 0
 \\ 0 & 0 & 0 & T
\end{bmatrix} \geq 0.
\] }
Hence, by using Lemma \ref{epsilpo}, we deduce
{\small
\begin{align*} & \begin{bmatrix}
 \Phi(P-Q)+ \varepsilon I & \Phi(P-Q)+ \varepsilon I \\ \Phi(P-Q)+ \varepsilon I & \Phi(P)+ \varepsilon I
\end{bmatrix} \\ &\qquad \geq \begin{bmatrix}
\Phi(0)+ \varepsilon I & 0 \\ \Phi(Q)+ \varepsilon I & 0
\end{bmatrix} \begin{bmatrix}
 \left( \Phi(Q)+ \varepsilon I\right)^{-1} & 0
 \\ 0 & T^{-1}
\end{bmatrix} \begin{bmatrix}
\Phi(0)+ \varepsilon I & \Phi(Q)+ \varepsilon I \\ 0 & 0
\end{bmatrix}.
 \end{align*}}
Therefore,
{\small
 \[ \begin{bmatrix}
 \Phi(P-Q)+ \varepsilon I & \Phi(P-Q)+ \varepsilon I \\ \Phi(P-Q)+ \varepsilon I & \Phi(P)+ \varepsilon I
\end{bmatrix} \geq \begin{bmatrix}
(\Phi(0)+ \varepsilon I) \left( \Phi(Q)+ \varepsilon I\right)^{-1}(\Phi(0)+ \varepsilon I) & \Phi(0)+ \varepsilon I\\ \Phi(0)+ \varepsilon I & \Phi(Q)+ \varepsilon I
\end{bmatrix}.
 \]}
 Consequently, the equality $\Phi(Q)= \Phi(P)$ implies that

 \[ \begin{bmatrix}
\star & \Phi(P-Q)-\Phi(0) \\ \Phi(P-Q)-\Phi(0) &0
\end{bmatrix} \geq 0,
 \]
in which $\star$ is a positive operator. Now, the positivity of the above matrix implies that $ \Phi(P-Q)= \Phi(0)$.
\end{proof}
We need the following lemma from \cite{DMK3}.
\begin{lemma}\cite[Lemma 2.4 (2)]{DMK3}\label{oneone2}
Let $ \Phi: \mathscr{A} \to \mathscr{B}$ be a $3$-positive map between unital $C^*$-algebras. If $\Phi(Z)=\Phi(0)$ for some positive invertible operator $Z\in \mathscr{A}$, then $\Phi(X)=\Phi(0)$ for every $X\in \mathscr{A}$.
\end{lemma}
If $\Phi: \bigoplus_{i=1}^k \mathscr{A}_i \to \mathscr{B}$ is a positive multilinear map and $\Phi(\textbf{I})=0$, then the above lemma ensures that $\Phi$ is identically zero. Indeed, if $\textbf{X}$ is a self-adjoint element of $\bigoplus_{i=1}^k \mathscr{A}_i$, then the restriction of $\Phi$ on the $C^*$-algebra generated by $\left\{\textbf{I}, \textbf{X} \right\}$ is a completely positive map. Hence, Lemma \ref{oneone2} ensures that $\Phi(\textbf{X})=0$. For an arbitrary element $\textbf{A}$ of $\bigoplus_{i=1}^k \mathscr{A}_i$, the assertion easily follows by using the self-adjoint decomposition of $\textbf{A}$ and applying the multilinearity of $\Phi$.\\
We say that a multimap $ \Phi: \bigoplus_{i=1}^k \mathscr{A}_i \to \mathscr{B}$ is tracial on $\mathscr{A}_j$ for some $ 1\leq j \leq k$ if
$\Phi(A_1, \ldots, AB, \ldots, A_k)= \Phi(A_1, \ldots, BA, \ldots, A_k)$ for all $A_i \in \mathscr{A}_i$ ($1\leq i\leq k, i\neq j$) and every $A,B \in \mathscr{A}_j$. Moreover, $\Phi$ is called tracial if it is tracial on $ \mathscr{A}_i$ for each $1\leq i\leq k$.

\begin{lemma}\label{3posit}
Let $\mathscr{M}$ be a properly infinite von Neumann algebra and let $\mathscr{B}$ be a unital $C^*$-algebra.
\begin{itemize}
\item[(i)]If $ \Phi: \mathscr{M} \to \mathscr{B}$ is a $3$-positive map such that $\Phi(X^*X)=\Phi(XX^*)$ for all $X\in \mathscr{M}$, then $\Phi$ is a constant map.
\item[(ii)] If $ \Phi: \mathscr{M} \to \mathscr{B}$ is a tracial linear map, then $\Phi$ identically zero.
\end{itemize}
\end{lemma}
 \begin{proof}
(i) Since $\mathscr{M}$ is properly infinite, according to \cite[Chapter V, Proposition 1.36]{TAK1}, there exists a projection $P \in \mathscr{M}$ such that $I-P\sim P \sim I$. Hence, there are some partial isometries $V, S\in \mathscr{M}$ such that $S^*S=V^*V=I$, $VV^*=P$, and $I-P=SS^*$. Therefore, by using the tracial assumption on $\Phi$, we get
 $\Phi(I)=\Phi(P)=\Phi(I-P)$. Using the equality $\Phi(I)=\Phi(P)$ and applying Lemma \ref{oneone}, we ensure that $\Phi(I-P) =\Phi(0)$. Consequently, $\Phi(I)=\Phi(0)$. Now, Lemma \ref{oneone2} implies that $\Phi(X)=\Phi(0)$ for every $X \in \mathscr{M}$.

(ii) If $\Phi$ is a tracial positive linear map, then the restriction of $\Phi$ on the $C^*$-algebra generated by $P$ and $I$ is a completely positive map. Hence, using part (i), we deduce that $\Phi(I)=\Phi(0)$, which implies the assertion.
 \end{proof}

Note that the operator norm of any $C^*$-algebra is an example of a $2$-positive map with $\| X^*X\|= \| XX^*\|$ for all $X$. Hence, in the nonlinear case, the $3$-positivity condition in Lemma \ref{3posit} is necessary.
\begin{remark}\label{remarksupport}
If $\mathscr{A}_1, \ldots, \mathscr{A}_k, \mathscr{B}$ are unital $C^*$-algebras such that $\mathscr{A}_j$ is a properly infinite von Neumann algebra for some $1\leq j\leq k$, then every positive mulitlinear map $ \Psi: \bigoplus_{i=1}^k \mathscr{A}_i \to \mathscr{B}$ that is tracial on $\mathscr{A}_j$, must be identically zero. Indeed, if we define the map $\Psi_{(j)} : \mathscr{A}_j \to \mathscr{B}$ by $ \Psi_{(j)}(A_j)= \Psi(I_1, \ldots, A_j, \ldots, I_k)$, then $\Psi_{(j)}$ is a tracial positive linear map. According to the second part of Lemma \ref{3posit}, we get
\[0= \Psi_{(j)}(I_j)= \Psi(I_1, \ldots, I_k).\]
 The result now follows by noting the discussion after Lemma \ref{oneone2}.
\end{remark}

Now we give a decomposition for tracial positive multilinear maps by employing some ideas of \cite{choi}. Noting a question raised in \cite[page 60]{choi}, we explore some suitable conditions to achieve a factorization.
For an algebra $\mathscr{X}$ and a subalgebra $\mathscr{Y}$ of $\mathscr{X}$, we denote the algebraic center of $\mathscr{Y}$ with respect to $\mathscr{X}$ by $\mathcal{Z}(\mathscr{Y})$.
\begin{theorem}\label{main22}
Let $ \Phi: \bigoplus_{i=1}^k \mathscr{A}_i \to \mathscr{B}$ be a tracial positive multilinear map between unital $C^*$-algebras. Suppose that one of the following properties is satisfied:
\begin{enumerate}
\item $\mathscr{B}$ is an injective $C^*$-algebra;
\item $\mathscr{B}$ is a von Neumann algebra;
\item The center of $\bigoplus_{i=1}^k \mathscr{A}_i$ (with respect to $\bigoplus_{i=1}^k \mathscr{A}_i$) is injective;
\item $\bigoplus_{i=1}^k \mathscr{A}_i$ is a von Neumann algebra;
\item $\Phi$ is weakly compact.
\end{enumerate}
Then there exist a compact Hausdorff space $\Omega$, a tracial positive linear map $\varphi_1 : \bigoplus_{i=1}^k \mathscr{A}_i\to C(\Omega)$, and a tracial positive multilinear map $\varphi_2 : C(\Omega) \to \mathscr{B}$ such that $\Phi=\varphi_2 \circ \varphi_1$. Moreover, if $\Phi$ is unital, then $\varphi_1$ and $\varphi_2$ can be chosen unital.
\end{theorem}
\begin{proof}
 Let $\tilde{\Phi}: \bigoplus_{i=1}^k \mathscr{A}_i^{**} \to \mathscr{B}^{**}$ be the positive multilinear separately weak$^*$-continuous extension of $\Phi$ mentioned in Lemma \ref{extension1}. It follows from the separate weak$^*$-continuity of $\tilde{\Phi}$ and the separate weak$^*$-continuity of the multiplication that $\tilde{\Phi}$ is also tracial.
 It follows from \cite[Chapter V, Theorem 1.19]{TAK1} that every projection in a von Neumann algebra is uniquely written as the sum of centrally orthogonal
projections $M$ and $N$ such that $M$ is finite and $N$ is properly infinite. Hence, since $\mathscr{A}_i^{**}$'s are von Neumann algebras, there are finite von Neumann algebras $\mathscr{M}_{i}$ and properly infinite von Neumann algebras $\mathscr{N}_{i}$ such that $\mathscr{A}_i^{**}=\mathscr{M}_{i} \oplus \mathscr{N}_{i}$, $1\leq i \leq k$. Let $C_{M_i}$ be the convex hull of $\{ U^*M_i U : U\in \mathscr{M}_{i}\text{ is unitary}\}$ in $\mathscr{M}_{i}$. In virtue of \cite[Proposition 2.4.5]{sakai}, $\mathcal{Z}(\mathscr{M}_{i}) \cap \left({\rm weak}^*-{\rm cl}\{{C}_{M_i}\}\right) $ is singleton for each $1\leq i \leq k$. Hence for every $1\leq i \leq k$, the map $\psi_i $ from $\mathscr{M}_{i}$ onto $\mathcal{Z}(\mathscr{M}_{i})$ can be defined by
\[\psi_i(M_i)\in \mathcal{Z}(\mathscr{M}_{i}) \cap \left({\rm weak}^*-{\rm cl}\{{C}_{M_i}\}\right) \qquad (M_i \in \mathscr{M}_{i}). \]
Moreover, $\psi_i$'s are center-valued tracial positive conditional expectations; see \cite[Theorem 2.4.6]{sakai}. It is easy to see that $\psi_i$'s ($1\leq i \leq k$) induce a tracial positive linear map
\begin{align*}\Gamma :\bigoplus_{i=1}^k \mathscr{A}_i^{**}=\bigoplus_{i=1}^k\left(\mathscr{M}_{i} \oplus \mathscr{N}_{i}\right) & \to \mathcal{Z}\left(\bigoplus_{i=1}^k \mathscr{M}_{i}\right) \subseteq \mathcal{Z}\left(\bigoplus_{i=1}^k \mathscr{A}_i^{**}\right) \\
(M_1\oplus N_1, \ldots, M_k\oplus N_k) &\mapsto (\psi_1(M_1), \ldots, \psi_k(M_k)).
\end{align*}
We prove the theorem when (1), (2) or (5) is valid. If $\mathscr{B}$ is injective, then the identity map ${\rm id}: \mathscr{B} \to \mathscr{B}$, as a completely positive linear map, induces a completely positive linear map $\Lambda : \mathscr{B}^{**} \to \mathscr{B}$ such that $\Lambda|_{\mathscr{B}}={\rm id}$. Hence, $\Psi:=\Lambda \circ \tilde{\Phi}$ is a tracial positive multilinear map from $\bigoplus_{i=1}^k \mathscr{A}_i^{**}$ to $\mathscr{B}$. If either $\mathscr{B}$ is a von Neumann algebra or $\Phi$ is weakly compact, according to Lemma \ref{extension1}, $\tilde{\Phi}$ is $\mathscr{B}$-valued. Hence, we may consider the extended map $\tilde{\Phi}$ as a map from $\bigoplus_{i=1}^k \mathscr{A}_i^{**}$ to $\mathscr{B}$ which is shown by $\Psi$. In all cases, $\Psi|_{\bigoplus_{i=1}^k\mathscr{A}_i}=\Phi$. Now, we shall show that $(\Psi \circ \Gamma)(A_1,\ldots, A_k) = \Psi(A_1,\ldots, A_k)$ for every $(A_1,\ldots, A_k)\in\bigoplus_{i=1}^k \mathscr{A}_i^{**}$. To this end, let us claim  that\\

\textbf{Claim ($\star$)}: If $\mathscr{M}_i$'s ($1\leq i\leq n$) are von Neumann algebras and $\mathscr{N}_i$'s ($1\leq i\leq n$) are properly infinite von Neumann algebras, then every tracial positive multilinear map $ \Phi: \bigoplus_{i=1}^n \left(\mathscr{M}_i \oplus \mathscr{N}_i\right)\to \mathscr{B}$ is completely determined by the restriction of $\Phi$ on $ \bigoplus_{i=1}^n \mathscr{M}_i$. More precisely, if $M_i\in \mathscr{M}_i$ and $N_i\in \mathscr{N}_i$ $ (1\leq i\leq n)$, then $\Phi (M_1\oplus N_1, \ldots, M_n\oplus N_n)= \Phi(M_1, \ldots, M_n)$.

To prove the claim, using the multilinearity of $\Phi$, we can write
\begin{align*}
\Phi( {M}_1\oplus {N}_1, \ldots, {M}_n\oplus {N}_n) &=\Phi( {M}_1, \ldots, {M}_n) \\ & \,\, \ + \Phi( {N}_1, {M}_2, \ldots, {M}_n) +\cdots+ \Phi( {N}_1, \ldots, {N}_n)\\
 & = \Phi( {M}_1, \ldots, {M}_n) + 0 +0+ \cdots + 0 \tag{by Remark \ref{remarksupport}} \\ & = \Phi( {M}_1, \ldots, {M}_n),
\end{align*}
which proves claim ($\star$).

According to the definition of $\Gamma$ and employing the fact that $\Psi$ is a tracial positive multilinear map, due to claim ($\star$), it is sufficient to verify that $$\Psi (\psi_1(M_1), \ldots, \psi_k(M_k))= \Psi(M_1, \ldots, M_k)$$ for all $(M_1, \ldots, M_k) \in \bigoplus_{i=1}^k\mathscr{M}_{i}$.

To achieve it, let $(M_1, \ldots, M_k) \in \bigoplus_{i=1}^k\mathscr{M}_{i}$ and $\sum_{j=1}^{p_i} \lambda_{i,j} U_{i,j} M_i U_{i,j}^*$ be arbitrary elements of $C_{M_i}$'s $(1\leq i \leq k)$ such that $\sum_{j=1}^{p_i} \lambda_{i,j}=1$. Then
\begin{align*}
&\Psi\left(\sum_{j=1}^{p_1} \lambda_{1,j} U_{1,j} M_1 U_{1,j}^*, \ldots, \sum_{j=1}^{p_k} \lambda_{k,j} U_{k,j} M_k U_{k,j}^*\right)\\
 &\qquad= \sum_{j_1=1}^{p_1}\sum_{j_2=1}^{p_2} \cdots \sum_{j_k=1}^{p_k} \lambda_{1,j_1} \cdots \lambda_{k,j_k}
 \Psi\left(U_{1,j_1} M_1 U_{1,j_1}^*, \ldots, U_{k,j_k} M_k U_{k,j_k}^*\right) \tag{since $\Psi$ is multilinear}
\\& \qquad = \sum_{j_1=1}^{p_1}\sum_{j_2=1}^{p_2} \cdots \sum_{j_k=1}^{p_k} \lambda_{1,j_1} \cdots \lambda_{k,j_k} \Psi( M_1, \ldots, M_k) \tag{since $\Psi$ is tracial} \\
 &\qquad= \Psi( M_1, \ldots, M_k). \
\tag{since $\sum_{j_s =1}^{p_s} \lambda_{j_s} = 1,\, 1\leq s\leq k$}
\end{align*}
It follows that $\Psi \circ \Gamma = \Psi$ on $\bigoplus_{i=1}^k C_{M_i}$. The assertion is concluded
by utilizing the weak$^*$-continuity of $\Psi$.

Next, note that the commutative $ C^*$-algebra $\mathcal{Z}\left(\bigoplus_{i=1}^k \mathscr{M}_{i}\right)$ can be identified with $C(\Omega)$ for some locally compact Hausdorff space $\Omega$.
Consider the following diagram:
 \[ \bigoplus_{i=1}^k \mathscr{A}_i^{**} \xrightarrow{\hspace*{.2cm}\Gamma\hspace*{.2cm}} C(\Omega)\xrightarrow{\hspace*{.2cm}\tilde{\Phi} |_{C(\Omega)}\hspace*{.2cm}} \mathscr{B}^{**} \xrightarrow{\hspace*{.2cm}\Lambda\hspace*{.2cm}} \mathscr{B}. \]
Taking $\varphi_1:= \Gamma|_\mathscr{A}$ and $\varphi_2 :=\Psi|_{C(\Omega)}$, the proof of Theorem in cases (1), (2), and (5) is completed.

Finally, if either $\mathcal{Z}\left(\bigoplus_{i=1}^k \mathscr{A}_i\right)$, considered as the center of $ \bigoplus_{i=1}^k \mathscr{A}_i$ with respect to $\bigoplus_{i=1}^k \mathscr{A}_i$, is injective or $\bigoplus_{i=1}^k \mathscr{A}_i$ is a von Neumann algebra, then there is a completely positive linear map $\Upsilon : \mathcal{Z}\left(\bigoplus_{i=1}^k \mathscr{A}_i^{**}\right) \to \mathcal{Z}\left(\bigoplus_{i=1}^k \mathscr{A}_i\right)$ such that $\Upsilon|_{\mathcal{Z}(\bigoplus_{i=1}^k \mathscr{A}_i)}={\rm id}$. Indeed, if $\mathcal{Z}\left(\bigoplus_{i=1}^k \mathscr{A}_i\right)$ is injective, then the identity map induces a completely positive map $\Upsilon : \mathcal{Z}\left(\bigoplus_{i=1}^k \mathscr{A}_i^{**}\right) \to \mathcal{Z}\left(\bigoplus_{i=1}^k \mathscr{A}_i\right)$. In the case when $ \bigoplus_{i=1}^k \mathscr{A}_i$ is a von Neumann algebra, since $\mathcal{Z}\left(\bigoplus_{i=1}^k \mathscr{A}_i\right)$ is also a von Neumann algebra, the assertion follows from \cite[Proposition 1.17.8]{sakai}; see also the previous part of the current proof. Let $\Gamma :\bigoplus_{i=1}^k \mathscr{A}_i^{**} \to \mathcal{Z}\left(\bigoplus_{i=1}^k \mathscr{A}_i^{**}\right)$ be the tracial positive linear map constructing in the previous part and let $\Theta : \mathcal{Z}\left(\bigoplus_{i=1}^k \mathscr{A}_i\right) \to \mathscr{B}$ be the restriction of $\Phi$ on $\mathcal{Z}\left(\bigoplus_{i=1}^k \mathscr{A}_i\right)$. Then we have the following diagram:
 \[ \bigoplus_{i=1}^k \mathscr{A}_i^{**} \xrightarrow{\hspace*{.2cm}\Gamma\hspace*{.2cm}} \mathcal{Z}\left(\bigoplus_{i=1}^k \mathscr{A}_i^{**}\right)\xrightarrow{\hspace*{.2cm}\Upsilon\hspace*{.2cm}} \mathcal{Z}\left(\bigoplus_{i=1}^k \mathscr{A}_i\right)\xrightarrow{\hspace*{.2cm}\Theta\hspace*{.2cm}} \mathscr{B}. \]
Setting $\varphi_1:= \Gamma|_\mathscr{A}$ and $\varphi_2:= \Theta \circ \Upsilon$
 and using a similar argument as in the proof of the previous part, we reach the desired assertion.
\end{proof}

\begin{corollary}\label{completcons}
If $ \Phi: \bigoplus_{i=1}^k \mathscr{A}_i \to \mathscr{B}$ is a tracial positive multilinear map between unital $C^*$-algebras and either of $\mathscr{B}$ or $\bigoplus_{i=1}^k \mathscr{A}_i$ has one of the properties (1), (2) ,(3), (4), and (5) in Theorem \ref{main22}, then $\Phi$ is completely positive. Moreover, for every $\textbf{P}\in\bigoplus_{i=1}^k \mathscr{A}_i^+$, the inducing positive multilinear map $\Psi : \bigoplus_{i=1}^k \mathscr{A}_i \to \mathscr{B}$ given by $\Psi_{\textbf{P}}(\textbf{A}) = \Phi(\textbf{PA})$ is completely positive.
\end{corollary}
\begin{proof}
The complete positivity of $\Phi$ is deduced from Theorem \ref{main22}. Indeed, if $\Phi = \varphi_2 \circ \varphi_1$ is the decomposition for $\Phi$ as mentioned in Theorem \ref{main22}, then the range of positive linear map $\varphi_1$ is commutative, so it is completely positive. Moreover, the domain of $\varphi_2$ is commutative; hence it is also a completely positive multilinear map. The above facts ensure that $\varphi_2 \circ \varphi_1$ is completely positive. To see the complete positivity of $\Psi_{\textbf{P}}$, let $\textbf{P}\in \bigoplus_{i=1}^k \mathscr{A}_i^+$. Then the tracial property of $\Phi$ implies that $\Psi(\textbf{A}) = \Phi\left(\textbf{P}^\frac{1}{2}\textbf{ A}\textbf{P}^\frac{1}{2}\right)$. Now let $[\textbf{A}_{i,j}]\in M_n\left(\bigoplus_{i=1}^k \mathscr{A}_i\right)^+\simeq \bigoplus_{i=1}^k M_n\left(\mathscr{A}_i\right)^+$. Then we can write
\[ 0\leq {\rm diag}_n (\textbf{P}^\frac{1}{2}, \ldots, \textbf{P}^\frac{1}{2}) [\textbf{A}_{i,j}]_{i,j=1}^n {\rm diag}_n(\textbf{P}^\frac{1}{2}, \ldots, \textbf{P}^\frac{1}{2})= \left[\textbf{P}^\frac{1}{2}\textbf{A}_{i,j}\textbf{P}^\frac{1}{2}\right]_{n\times n},\]
where ${\rm diag}_n ({X}_1, \ldots, {X}_n)$ is used to denote the $n\times n$ diagonal matrix $ {X}_1\oplus \cdots \oplus {X}_n$.
Since $\Phi$ is completely positive, we get
$$0\leq \Phi_n \left(\left[\textbf{P}^\frac{1}{2}\textbf{A}_{i,j}\textbf{P}^\frac{1}{2}\right]\right)= \Psi_n \left(\left[\textbf{A}_{i,j}\right]\right)$$
for every $n\in \mathbb{N}$.
\end{proof}

Our next assertion gives a version of the Choi--Tsui result for nonlinear tracial completely positive maps.
\begin{theorem}\label{th2.11}
Let $\mathscr{M}$ be a von Neumann algebra and let $\mathscr{B}$ be a unital $C^*$-algebra. If $\Phi: \mathscr{M} \to \mathscr{B}$ is a tracial completely positive map, then there exist a compact Hausdorff space $\Omega$, a tracial positive linear map $\varphi_1 :\mathscr{M} \to C(\Omega)$, and a tracial completely positive map $\varphi_2 : C(\Omega) \to \mathscr{B}$ such that $\Phi=\varphi_2 \circ \varphi_1$. Moreover, if $\Phi$ is unital, then $\varphi_1$ and $\varphi_2$ can be chosen unital.
\end{theorem}
\begin{proof}
As mentioned in Subsection \ref{overviewnonlinear}, there are $(m,n)$-mixed homogeneous completely positive maps
 $\Phi_{m,n} : \mathscr{M}\to \mathscr{B}$ such that
\[ \Phi (A)= \Phi(0) + \sum _{\substack{m,n=0\\ m+n\geq 1}}^\infty \Phi_{m,n}(A) \qquad (A\in \mathscr{A}).\] In addition,
\[ \Phi_{m,n}(A)=\frac{1}{m! n!} \frac{\partial^{m+n}}{\partial^m z \partial^n \bar{z}} \Phi(zA)|_{z=0} \qquad (z\in \mathbb{C}).\]
Hence, if $\Phi$ is tracial, then $\Phi_{m,n}$'s are tracial ($m, n= 0, 1, \ldots$).
Moreover, if $m+n>0$, then for every $(m,n)$-mixed homogeneous completely positive map $\Phi_{m,n} : \mathscr{M}\to \mathscr{B}$, there is a completely positive $(m+n)$-linear map $\Phi^{(m,n)}: \mathscr{M}^{(m)}\bigoplus \bar{\mathscr{M}}^{(n)} \to \mathscr{B}$ such that $\Phi_{m,n}(A)= \Phi^{(m,n)}(A^{(m)} \oplus \bar{A}^{(n)})$, where $\bar{\mathscr{M}}$ is the $C^{*}$-algebra conjugate to $\mathscr{M}$. Note that $\mathscr{M}^{(0)}=\bar{\mathscr{M}}^{(0)}:=\{0\}$. Let $\mathscr{M}=\mathscr{M}_1 \oplus \mathscr{M}_2$, where $\mathscr{M}_1$ is a finite von Neumann algebra and $ {\mathscr{M}_2}$ is a properly infinite von Neumann algebra. Then it is easy to check that $\bar{\mathscr{M}}=\bar{\mathscr{M}}_1 \oplus \bar{ \mathscr{M}}_2$. A similar argument as in the proof of Theorem~\ref{main22} can be used to show that there is a tracial positive linear map $\psi: \mathscr{M}_1 \to \mathcal{Z}\left( \mathscr{M}_1\right)$ (resp., $\tilde{\psi}: \tilde{\mathscr{M}}_1 \to \mathcal{Z}(\tilde{\mathscr{M}}_1)$) such that
$\psi(D)=\mathcal{Z}(\mathscr{M}) \cap \left({\rm weak}^*-{\rm cl}\{{C}_{D}\}\right) $ (resp., $\tilde{\psi}(E)=\mathcal{Z}(\bar{\mathscr{M}}) \cap \left({\rm weak}^*-{\rm cl}\{{C}_{E}\}\right) $)
for $D\in \mathscr{M}_1 $ (resp., $E \in \bar{\mathscr{M}}_1 $), where ${C}_{D}$ (resp., ${C}_{E}$) are the convex hull of $\{ U^* D U : U\in \mathscr{M}_1\text{ is unitary}\}$ in $\mathscr{M}_1$ (resp., the convex hull of $\{ V^* E V : V\in \bar{\mathscr{M}}_1\text{ is unitary}\}$ in $\bar{\mathscr{M}}_1$).
 Now define the map
\[
\Gamma : \mathscr{M}^{\oplus_{m=1}^\infty} \bigoplus \bar{\mathscr{M}}^{\oplus_{n=1}^\infty} \to \mathcal{Z}\left( \mathscr{M}\right)^{\oplus_{m=1}^\infty}\bigoplus \mathcal{Z}\left( \bar{\mathscr{M}}\right)^{\oplus_{n=1}^\infty}
\]
by
{\small\begin{align*}
\Gamma\left( (A_1, A_2, \ldots ) \oplus (B_1, B_2, \ldots )\right)
& = \Gamma \left((E_1\oplus F_1, E_2\oplus F_2, \ldots ) \oplus (D_1\oplus G_1, E_2\oplus G_2, \ldots )\right) \\
 &= \left((\psi(D_1), \psi(D_2), \ldots ) \oplus (\bar{\psi}(E_1), \bar{\psi}( E_2), \ldots )\right),
\end{align*}}
where $D_i \in \mathscr{M}_1, F_i \in \mathscr{M}_2, E_i \in \bar{\mathscr{M}}_1$, and $G_i \in \bar{\mathscr{M}}_2, i=1, 2,\ldots $.
Clearly, $\Gamma$ is a tracial positive linear map with the commutative range that ensures $\Gamma$ to be a completely positive map. Let $\iota : \mathscr{M}\to \mathscr{M}^{\oplus_{m=1}^\infty}\bigoplus \bar{\mathscr{M}}^{\oplus_{n=1}^\infty}$ be the embedding defined by $\iota (A)=(A, A, \ldots) \oplus (\bar{A}, \bar{A}, \ldots )$, where $\bar{A}$ is the
element in $\bar{\mathscr{M}}$ corresponding to $A\in \mathscr{M}$. Then obviously $\iota$ is a completely positive map. Therefore, we have a completely positive linear map $\Gamma \circ \iota : \mathscr{M} \to
\mathcal{Z}\left( \mathscr{M}\right)^{\oplus_{m=1}^\infty}\bigoplus \mathcal{Z}\left( \bar{\mathscr{M}}\right)^{\oplus_{n=1}^\infty}$. Moreover, for each pair of nonnegative integers $m$ and $ n$ with $m+n>0$, let $\Pi_{m,n}:\mathscr{M}^{\oplus_{m=1}^\infty}\bigoplus \bar{\mathscr{M}}^{\oplus_{n=1}^\infty}\to \mathscr{M}^{(m)}\bigoplus \bar{\mathscr{M}}^{(n)} $ be the mixed projection defined by \[\Pi_{m,n} \left((A_1, A_2, \ldots ) \oplus (B_1, B_2, \ldots )\right)= (A_1, A_2, \ldots, A_m, B_1, B_2, \ldots, B_n).\]
 Then $\Pi_{m,n}$ is a completely positive linear map. Define
$$\varphi_1 : \mathscr{A}\to \mathcal{Z}\left( \mathscr{A}\right)^{\oplus_{m=1}^\infty}\bigoplus \mathcal{Z}\left( \bar{\mathscr{A}}\right)^{\oplus_{n=1}^\infty} \, \, \mbox{and}\,\,
 \varphi_2 : \mathcal{Z}\left( \mathscr{A}\right)^{\oplus_{m=1}^\infty}\bigoplus \mathcal{Z}\left( \bar{\mathscr{A}}\right)^{\oplus_{n=1}^\infty} \to \mathscr{B}$$
by
\[ \varphi_1:=\Gamma \circ \iota \qquad \text{and} \qquad \varphi_2 :=\Phi(0)+ \sum _{\substack{m,n=0\\ m+n\geq 1}}^\infty \left(\Phi^{(m,n)}|_{\mathcal{Z}(\mathscr{A})^{\oplus_{m}}\bigoplus \mathcal{Z}( \bar{\mathscr{A}})^{\oplus_{n} }} \circ \Pi_{m,n}\right).\]
Then $\varphi_1$ is a tracial completely positive linear map and $ \varphi_2 $ is a tracial completely positive map. It remains to prove $\varphi_2 \circ \varphi_1 = \Phi$.
Let $A\in \mathscr{A}$ and $A=M\oplus N$, where $M\in \mathscr{M}_1$ and $N\in \mathscr{M}_2$. Then we have
$$\iota(A)= \left((A, A, \ldots ) \oplus (\bar{A}, \bar{A}, \ldots )\right) = ( M\oplus M, M\oplus N, \ldots) \oplus (\bar{M}\oplus \bar{N}, \bar{M}\oplus \bar{N}, \ldots),$$
 so
\begin{align*}
\varphi_1(A) &= (\Gamma \circ \iota) ( M\oplus N, M\oplus N, \ldots) \oplus (\bar{M}\oplus \bar{N}, \bar{M}\oplus \bar{N}, \ldots)
\\ & = \left(\psi(M), \psi(M), \ldots \right) \oplus \left(\bar{\psi}(\bar{M}), \bar{\psi}( \bar{M}), \ldots \right),
\end{align*}
where $\psi(M)$ and $\bar{\psi}( \bar{M})$ belong to the weak$^*$-closed convex hull of $C_{M}$ and $C_{\bar{M}}$, respectively.
 Let
{\small \[\left( \sum_{j=1}^{p_1}\lambda_{1,j} U^*_{1,j}M U_{1,j}, \sum_{j=1}^{p_2}\lambda_{2,j} U^*_{2,j}M U_{2,j}, \ldots \right) \oplus \left(\sum_{j=1}^{q_1}\mu_{1,j} V^{*}_{1,j}\bar{M} V_{1,j}, \sum_{j=1}^{q_2}\mu_{2,j} V^{*}_{2,j}\bar{M} V_{2,j}, \ldots \right) \]}
be an element of $\in C_{M}^{\oplus_{m=1}^\infty}\bigoplus C_{\bar{M}}^{\oplus_{n=1}^\infty}$.
For fixed nonnegative integers $m$ and $ n$ with $m+n>0$, we have
\begin{footnotesize}
\begin{align*}
&\left(\Phi^{(m,n)}\circ \Pi_{m,n}\right) \left( \left( \sum_{j=1}^{p_1}\lambda_{1,j} U^*_{1,j}M U_{1,j}, \ldots \right) \oplus \left(\sum_{j=1}^{q_1}\mu_{1,j} V^{*}_{1,j}\bar{M} V_{1,j},\ldots \right) \right) \\
& = \Phi^{(m,n)}\left( \sum_{j=1}^{p_1}\lambda_{1,j} U^*_{1,j}M U_{1,j}, \ldots,\sum_{j=1}^{p_m}\lambda_{m,j} U^*_{m,j}M U_{m,j}, \sum_{j=1}^{q_1}\mu_{1,j} V^{*}_{1,j}\bar{M} V_{1,j},\ldots, \sum_{j=1}^{q_n}\mu_{n,j} V^{*}_{n,j}\bar{M} V_{n,j}\right)
\\ & = \Phi^{(m,n)}\left( M^{(m)}, \bar{M}^{(n)}\right)
 \tag{Since $\Phi^{(m,n)}$ is a tracial positive multilinear map}
\\ & = \Phi_{m,n}(M).
\end{align*}
\end{footnotesize}
As
 $\left(\Phi^{(m,n)}|_{\mathcal{Z}(\mathscr{A})^{\oplus_{m}}\bigoplus \mathcal{Z}( \bar{\mathscr{A}})^{\oplus_{n} }} \circ \Pi_{m,n}\right)$'s are separate weak$^*$-continuous for all $m, n =0, 1, \ldots$ (by Lemma \ref{arbexte}), we get
 \[ \left(\Phi^{(m,n)}|_{\mathcal{Z}(\mathscr{A})^{\oplus_{m}}\bigoplus \mathcal{Z}( \bar{\mathscr{A}})^{\oplus_{n} }} \circ \Pi_{m,n}\circ \varphi_1\right) (M)= \Phi_{m,n} (M)
 \]
 for each $M \in \mathscr{M}_1$ and nonnegative integers $m$ and $n$ with $m+n>0$.
 Hence, by applying claim ($\star$) in the proof of Theorem \ref{main22}, we ensure that $\varphi_2 \circ \varphi_1(A)= \Phi(A)$.

\end{proof}

 \section{Applications in Quantum Information Theory}
 An important topic in quantum information theory is to study the covariance and variance concerning with a positive map between $C^*$-algebras. To see more about this topic and its development, we refer the reader to \cite{DM2,matias, petz,yangi}.
In this direction, a basic problem, which is of particular importance in statistical and quantum mechanics, is the variance-covariance inequality. In noncommutative frameworks, it is equivalent to the positivity of the so-called variance-covariance matrix \begin{footnotesize}
$ \begin{bmatrix}
{\rm Var}(A) & {\rm Cov}(A,B)\\
{\rm Cov}(B,A) & {\rm Var}(B)
\end{bmatrix}
$\end{footnotesize}; see \cite{ARAM, bht, matias}.
 We intend to extend the variance and covariance for positive multimaps and give some noncommutative uncertainty relations for them. In this regard, we have at least two choices.
 In the first perspective, since a (not necessarily linear or multilinear) positive multimap $\Phi:\bigoplus_{i=1}^k \mathscr{A}_i\to \mathscr{B}$ between $C^*$-algebras can be considered as a positive map from the $C^*$-algebra $\bigoplus_{i=1}^k \mathscr{A}_i$ to the $C^*$-algebra $\mathscr{B}$,
 the covariance and variance of $\textbf{A},\textbf{B}\in\bigoplus_{i=1}^k \mathscr{A}_i$ with respect to $\Phi$ can be naturally defined by
 $ {\rm Cov}_\Phi (\textbf{A},\textbf{B})= \Phi(\textbf{A}^*\textbf{B})- \Phi( \textbf{A}^*)\Phi(\textbf{B})$ and
 ${\rm Var}_\Phi(\textbf{A})= {\rm Cov}_\Phi (\textbf{A},\textbf{A})$, respectively.
 Matias \cite{matias} showed that if $\Phi$ is a unital $3$-positive map, then the variance-covariance matrix
 \[ \begin{bmatrix}
{\rm Var}_\Phi(\textbf{A}) & {\rm Cov}_\Phi(\textbf{A},\textbf{B})\\
{\rm Cov}_\Phi(\textbf{B},\textbf{A}) & {\rm Var}_\Phi(\textbf{B})
\end{bmatrix}
\]
is positive for all $\textbf{ A, B }\in \bigoplus_{i=1}^k \mathscr{A}_i$.

If $\textbf{A}$ and $\textbf{B}$ are self-adjoint operators (observables in quantum information), then it is worth mentioning that in the case that the range of $\Phi$ is commutative, the Heisenberg and the Schr\"{o}dinger uncertainty relations are just a result of the variance-covariance inequality. Indeed,
{\small
 \begin{align*} 0 &\leq \begin{bmatrix}
{\rm Var}_\Phi(\textbf{A}) & {\rm Cov}_\Phi(\textbf{A},\textbf{B})\\
{\rm Cov}_\Phi(\textbf{B},\textbf{A}) & {\rm Var}_\Phi(\textbf{B})
\end{bmatrix} \\ &= \begin{bmatrix}
{\rm Var}_\Phi(\textbf{A}) & {\rm Re Cov}_\Phi(\textbf{A},\textbf{B}) + {i} {\rm Im Cov}_\Phi(\textbf{A},\textbf{B}) \\
{\rm Re Cov}_\Phi(\textbf{B},\textbf{A}) + {i} {\rm Im Cov}_\Phi(\textbf{B},\textbf{A})& {\rm Var}_\Phi(\textbf{B})
\end{bmatrix}
.\end{align*}}
If the range of $\Phi$ is commutative, then
{\small
\begin{align*}
 {\rm Im Cov}_\Phi(\textbf{A},\textbf{B}) = \dfrac{\Phi(\textbf{AB})-\Phi(\textbf{BA})}{2i}
\, \quad \mbox{and}\quad\,
 {\rm Im Cov}_\Phi(\textbf{B},\textbf{A})=\dfrac{\Phi(\textbf{BA})-\Phi(\textbf{AB})}{2i}.
\end{align*}}
Hence,
{\small
\[
 \begin{bmatrix}
{\rm Var}_\Phi(\textbf{A}) & {\rm Re Cov}_\Phi(\textbf{A},\textbf{B}) + \frac{1}{2} \left(\Phi(\textbf{AB})-\Phi(\textbf{BA})\right) \\
{\rm Re Cov}_\Phi(\textbf{B},\textbf{A}) + \frac{1}{2} \left(\Phi(\textbf{BA})-\Phi(\textbf{AB})\right) & {\rm Var}_\Phi(\textbf{B})
\end{bmatrix} \geq 0.
\]
}
It follows from the positivity of the above matrix and the commutativity of the range of $\Phi$ that
{\small
\begin{align}
 {\rm Var}_\Phi(\textbf{A}) {\rm Var}_\Phi(\textbf{B}) &\geq \left|{\rm Re Cov}_\Phi(\textbf{A},\textbf{B}) + \frac{1}{2} \left(\Phi(\textbf{AB})-\Phi(\textbf{BA})\right)\right|^2 \nonumber \\ &= |{\rm Re Cov}_\Phi(\textbf{A},\textbf{B})|^2 + \frac{1}{4} \left|\Phi(\textbf{AB})-\Phi(\textbf{BA})\right|^2.
 \end{align}}
If $\Phi$ is linear, then the above inequality turns into
{\small
\begin{align}\label{uncert1}
 {\rm Var}_\Phi(\textbf{A}) {\rm Var}_\Phi(\textbf{B}) &\geq |{\rm Re Cov}_\Phi(\textbf{A},\textbf{B})|^2 + \frac{1}{4} \left|\Phi\left([\textbf{A},\textbf{B}]\right)\right|^2,
 \end{align}}
where $[X,Y]=XY-YX$ is the commutator of operators $X$ and $Y$. Inequality \eqref{uncert1} is just a generalization for the classical Schr\"{o}dinger uncertainty relation; see \cite{Li, DM2, yangi}.
Moreover, some versions of uncertainty relation inequalities, including a generalized variance and covariance in noncommutative setting, can be found in \cite{DM2}.

In this section, we aim to give some generalizations of the above inequalities for positive multilinear maps and completely positive maps between $C^*$-algebras as some applications of Theorems \ref{main22} and \ref{th2.11}. Note that if $\Phi:\mathscr{A} \to \mathscr{B}$ is a unital $m$-positive map (resp., a tracial $m$-positive map) and $P$ is a density operator, that is, $P\in \mathscr{A}^+$ and $\Phi(P)=I$, then the map $\Psi:\mathscr{A} \to \mathscr{B}$ defined by $\Psi(X)=\Phi(P^\frac{1}{2}XP^\frac{1}{2})$ (resp., $\Psi(X)=\Phi(PX)$) is also a unital $m$-positive map. Thus, we explore properties of such positive maps $\Phi$ and then apply them for maps such as $\Psi$, which are frequently appeared in quantum mechanics. We simply denote the map $\Psi$ by $\Phi(P^\frac{1}{2}\square P^\frac{1}{2})$ (resp., $\Phi(P\square )$).

It is known that if $\Phi:\mathscr{A} \to \mathscr{B}$ is a non-necessarily linear $3$-positive map with $\Phi(0)=0$, then $\Phi$ is superadditive on $\mathscr{A}^+$; see \cite[Theorem 3.5]{DM3}. In particular, if $A\geq B\geq 0$, then
\begin{equation}\label{super}
\Phi(A-B)\leq\Phi(A)-\Phi(B).
\end{equation}
Moreover, if $\Phi:\mathscr{A} \to \mathscr{B}$ is a $2m$-positive map, then $\Phi_2$ is monotone on $M_m(\mathscr{A})$.
\begin{definition}\label{decom}
We say that a unital (tracial) positive map $\Phi:\mathscr{A} \to \mathscr{B}$ between unital $C^*$-algebras enjoys the property $\mathscr{D}_m$ ($1\leq m\leq \infty$), if there are a compact Hausdorff space $\Omega$, a unital (tracial) positive linear map $\psi:\mathscr{A}\to C(\Omega)$, and a unital (tracial) $m$-positive (in the case when $m=\infty$, a completely positive) map $\varphi : C(\Omega) \to \mathscr{B}$ such that $\Phi= \varphi \circ \psi$.
\end{definition}
 Obviously, if $\Phi$ enjoys the property $\mathscr{D}_m$, then it is $m$-positive. Moreover, if $\Phi : \bigoplus_{i=1}^k \mathscr{A}_i \to \mathscr{B}$ is a tracial positive multilinear map satisfying one of the four properties in Theorem~\ref{main22}, then $\Phi$ and $\Phi(\textbf{P}\square)$, when $\textbf{P}\in\bigoplus_{i=1}^k \mathscr{A}_i$ is a density operator, have the property $\mathscr{D}_\infty$. Furthermore, Theorem~\ref{th2.11} shows that if $\mathscr{A}$ is a von Neumann algebra and $\mathscr{B}$ is a unital $C^*$-algebra, then every tracial completely positive map $\Phi: \mathscr{A} \to \mathscr{B}$ has the property $\mathscr{D}_\infty$ and also holds for $\Phi(P\square)$, when $P\in \mathscr{A}$ is a density operator.
\begin{theorem}\label{hiesenberg}
 Let $\Phi : \mathscr{A} \to \mathscr{B}$ be a positive map between unital $C^*$-algebras with $\Phi(0)=0$, and let $\Phi$ enjoy the property $\mathscr{D}_m$ for some $m\geq 3$. If ${A}, {B}\in\mathscr{A} $ are self-adjoint, then the following statements hold:
 \begin{itemize}
 \item[(i)] The variance-covariance matrix is positive.
\item[(ii)] An uncertainty relation holds as follows:
{\small
\[ \begin{bmatrix}
{\rm Var}_{\Phi}( {A}) & \Phi\left( \frac{1}{2}[ {A}, {B}]\right) \\
 \Phi\left(\frac{1}{2}[ {B}, {A}]\right) & {\rm Var}_{\Phi}( {B})
\end{bmatrix} \geq 0.\]}
{\small
\begin{align*}
\mathrm{(iii)} \ \left\|{\rm Var}_{\Phi}( {A})\right\| {\rm Var}_{\Phi}( {B})
\geq \left|\Phi\left( \frac{1}{2}\left[ {A}, {B}\right]\right)\right|^2
\quad\mbox{and}\quad
 \left\|{\rm Var}_{\Phi}( {B})\right\|{\rm Var}_{\Phi}( {A})
 \geq \left|\Phi\left( \frac{1}{2}\left[ {A}, {B}\right]\right)\right|^2.
 \end{align*}}
In particular, if the range of $\Phi$ is commutative, then
{\small\begin{align}\label{ki2}
 {\rm Var}_{\Phi}( {A}) {\rm Var}_{\Phi}( {B}) \geq \left|\Phi\left(\frac{1}{2}\left[ {A}, {B}\right]\right)\right|^2.
\end{align}}
\item[(iv)] If $\Phi$ enjoys the property $\mathscr{D}_m$ for some $m\geq 4$, then
{\small
\[ \begin{bmatrix}
{\rm Var}_{\Phi}( {A}) & \Phi(\frac{1}{2}\{ {A}, {B}\}) -\Phi( {A})\Phi( {B}) \\
 \Phi(\frac{1}{2}\{ {A}, {B}\}) -\Phi( {B})\Phi( {A}) & {\rm Var}_{\Phi}( {B})
\end{bmatrix} \geq 0,\]}
where $\{A,B\}=AB+BA$.
 \end{itemize}
\end{theorem}
\begin{proof}
It follows from the assumption that $\Phi$ is a $3$-positive map. Hence, (i) immediately follows (see the descriptions at the beginning of the section). Let $\psi: \mathscr{A} \to C(\Omega)$ and $\varphi :C(\Omega) \to \mathscr{B}$ be a completely positive linear and an $m$-positive maps, respectively, decomposing $\Phi$ as in Definition \ref{decom}. Since $\psi$ is a completely positive linear map and the range of $\psi$ is commutative, it follows from inequality \eqref{uncert1} that
{\small
 \begin{equation} \begin{bmatrix}\label{avali}
{\rm Var}_{\psi}( {A}) &\psi\left( \frac{1}{2}[ {A}, {B}]\right) \\
\psi\left( \frac{1}{2}[ {B}, {A}]\right) & {\rm Var}_{\psi}( {B})
\end{bmatrix} \geq 0
 \end{equation}}
 and
 {\small\begin{equation}\label{dovomi}
 \begin{bmatrix}
{\rm Var}_{\psi}( {A}) & {\rm Re Cov}_{\psi}( {A}, {B}) \\
{\rm Re Cov}_{\psi}( {B}, {A}) & {\rm Var}_{\psi}( {B})
\end{bmatrix} \geq 0.
\end{equation}}
We first show that
\begin{equation}\label{varicomp}
\varphi \left({\rm Var}_{\psi}( {X})\right) \leq {\rm Var}_{\Phi}\left( {X} \right)
\end{equation}
 for all self-adjoint elements $ {X}\in \mathscr{A}$. Since $\varphi$ is a $3$-positive map, we have
 {\small\begin{align*}
 {\rm Var}_{\Phi}\left( {X}\right)&=\Phi( {X}^2)- \Phi( {X})^2= \varphi\left(\psi\left( {X}^2\right)\right) -\big(\varphi(\psi \left( {X}\right)) \big)^2 \\ &\geq \varphi\left(\psi\left( {X}^2\right)\right) -\varphi \left(\psi ( {X})^2\right) \tag{by using the Choi type inequality \eqref{kadison} for $\varphi$}
 \\ & \geq \varphi\left(\psi( {X}^2) - \psi ( {X})^2\right) \tag{since $\psi( {X}^2) \geq \psi ( {X})^2\geq 0$, by using inequality \eqref{super} for $\varphi$}\\
 &= \varphi\left({\rm Var}_{\psi}( {X})\right).
\end{align*}}
 Since $\varphi$ is $3$-positive, by using \eqref{avali} and \eqref{varicomp}, we get
{\small \begin{align}
 0\leq \varphi_2\left( \begin{bmatrix}
{\rm Var}_{\psi}( {A}) &\psi\left( \frac{1}{2}[ {A}, {B}]\right) \\
\psi\left( \frac{1}{2}[ {B}, {A}]\right) & {\rm Var}_{\psi}( {B})
\end{bmatrix}\right) &= \begin{bmatrix}
\varphi\left({\rm Var}_{\psi}( {A})\right) &\Phi\left( \frac{1}{2}[ {A}, {B}]\right) \\
\Phi\left( \frac{1}{2}[ {B}, {A}]\right) & \varphi\left({\rm Var}_{\psi}( {B})\right)
\end{bmatrix} \nonumber\\ &\leq \begin{bmatrix}
{\rm Var}_{\Phi}( {A})& \Phi\left( \frac{1}{2}[ {A}, {B}]\right) \\
\Phi\left( \frac{1}{2}[ {B}, {A}]\right) & {\rm Var}_{\Phi}( {B}),
\end{bmatrix} \tag{by \eqref{varicomp}}
\end{align}}
which proves (ii). Next, note that the positivity of
{\small \begin{equation}\label{kia6}
\begin{bmatrix}
{\rm Var}_{\Phi, P}( {A})& \Phi\left( \frac{1}{2}[ {A}, {B}]\right) \\
\Phi\left( \frac{1}{2}[ {B}, {A}]\right) & {\rm Var}_{\Phi}( {B})
\end{bmatrix}
 \end{equation}}
 ensures that both matrices
 {\small \begin{equation}
 \begin{bmatrix}
{\rm Var}_{\Phi}( {A})& \Phi\left( \frac{1}{2}[ {A}, {B}]\right) \\
\Phi\left( \frac{1}{2}[ {B}, {A}]\right) & \|{\rm Var}_{\Phi}( {B})\|
\end{bmatrix} \qquad \text{and}\qquad \begin{bmatrix}
\|{\rm Var}_{\Phi, P}( {A})\|& \Phi\left( \frac{1}{2}[ {A}, {B}]\right) \\
\Phi\left( \frac{1}{2}[{B}, {A}]\right) & {\rm Var}_{\Phi}( {B})
\end{bmatrix}
 \end{equation}}
are positive. Applying Lemma \ref{epsilpo} and noting that $|\Phi\left( \frac{1}{2}[{B}, {A}]\right)|^2=|\Phi\left( \frac{1}{2}[{A}, {B}]\right)|^2$, we deduce (iii).
Moreover, If the range of $\Phi$ is commutative, then the positivity of matrix \eqref{kia6} immediately implies \eqref{ki2}.

 To see (iv), first note that the positivity of the matrix \eqref{dovomi} implies that
{\small
 \begin{equation}\label{sevomi}
 \begin{bmatrix}
\psi( {A}^2) & \psi\left(\frac{1}{2}\{ {A}, {B}\}\right) \\
\psi\left(\frac{1}{2}\{ {A}, {B}\}\right) & \psi( {B}^2)
\end{bmatrix} \geq \begin{bmatrix} \psi( {A})^2 & \psi( {A})\psi( {B}) \\
\psi( {B})\psi( {A}) & \psi( {B})^2
\end{bmatrix}.
\end{equation}}
To the right of the above inequality, an easy computation shows that
{\small
\[\begin{bmatrix} \psi( {A})^2 & \psi( {A})\psi( {B}) \\
\psi( {B})\psi( {A}) & \psi( {B})^2
\end{bmatrix}= \begin{bmatrix} \psi( {A}) \\
 \psi( {B})
\end{bmatrix} \begin{bmatrix} \psi( {A}) \\
 \psi( {B})
\end{bmatrix}^*,\]}
which ensures that both of matrices in inequality \eqref{sevomi} are positive. Since $\varphi$ is $4$-positive,
 $\varphi_2$ is $2$-monotone on $M_2(\mathscr{A})_+$. Hence,
{\small \begin{align*}
\begin{bmatrix}
\Phi( {A}^2) & \Phi\left(\frac{1}{2}\{ {A}, {B}\}\right) \\
\Phi\left(\frac{1}{2}\{ {A}, {B}\}\right) & \Phi( {B}^2)
\end{bmatrix}&=
\varphi_2\left(\begin{bmatrix}
\psi( {A}^2) & \frac{1}{2}\psi(\{ {A}, {B}\}) \\
\frac{1}{2}\psi(\{ {A}, {B}\}) & \psi( {B}^2)
\end{bmatrix}\right) \\ &\geq \varphi_2\left( \begin{bmatrix} \psi( {A})^2 & \psi( {A})\psi( {B}) \\
\psi( {B})\psi( {A}) & \psi( {B})^2
\end{bmatrix}\right) \\ \tag{Since $\varphi_2$ is monotone on positive elements} \\ &=\varphi_2\left( \begin{bmatrix} \psi( {A}) \\
 \psi( {B})
\end{bmatrix} \begin{bmatrix} \psi( {A}) \\
 \psi( {B})
\end{bmatrix}^*\right)
\\ & \geq
\varphi_2\left( \begin{bmatrix} \psi( {A}) \\
 \psi( {B})
\end{bmatrix}\right) \varphi_2\left( \begin{bmatrix} \psi( {A}) \\
 \psi( {B})
\end{bmatrix}^*\right)
 \\ \tag{by applying the Choi type inequality \eqref{kadison} for $\varphi_2$}\\ &=
 \begin{bmatrix} \Phi( {A}) \\
 \Phi( {B})
\end{bmatrix} \begin{bmatrix} \Phi( {A}) \\
 \Phi( {B})
\end{bmatrix}^*
\\ &= \begin{bmatrix} \Phi( {A})^2 & \Phi( {A})\Phi( {B}) \\
\Phi( {B})\Phi( {A}) & \Phi( {B})^2
\end{bmatrix},
\end{align*}}
which gives (iii).
\end{proof}

\subsection{Some Applications in Skew Information}
The skew entropy and the problem of its convexity are extensively studied in the entropy theory. The skew entropy can be considered as a kind of the degree for non-commutativity between a quantum state (density operator) and an observable (self-adjoint operator).
 If $A$ and $B$ are self-adjoint elements and $\varrho$ is a density operator, then a one-parameter correlation and the Wigner--Yanase skew information (is known as the Wigner--Yanase--Dyson skew information) for $A$ and $B$ are defined by
 \begin{eqnarray*}
{\rm Corr}_{\varrho}^{\alpha}(A,B):={\rm Tr}(\varrho AB)-{\rm Tr}(\varrho^{1-\alpha} A\varrho^{ \alpha}B) \mbox{\quad and \quad} {\rm I}_{\varrho}^{\alpha}(A):={\rm Corr}_{\varrho}^{\alpha}(A,A),
\end{eqnarray*}
 respectively, where $\alpha \in [0,1]$; see \cite{yangi}. It is known that
 \begin{eqnarray}\label{skew1}
\left| {\rm Re(Corr}_\varrho^{\alpha} (A,B))\right|^2 \leq {\rm I}_{\varrho}^\alpha (A){\rm I}_{\varrho}^\alpha (B).
\end{eqnarray}
Two continuous real-valued functions $f$ and $g$ defined on a set $D$ are called of same monotonic if $(f(x)-f(y)) (g(x)-g(y)) \geq 0$ for every $x,y \in D$. The authors of \cite{DM2} gave a generalization for the Wigner--Yanase--Dyson skew information as follows. Let $\Phi : \mathscr{A}\to \mathscr{B}$ be a tracial positive linear map between von Neumann algebras, and let $\varrho\in \mathscr{A}$ be a density operator. For each pair of same monotonic functions $f$ and $g$, which are continuous real-valued functions defining on an interval containing the spectrum of $\varrho$, the generalized correlation and the generalized Wigner--Yanase--Dyson skew information of self-adjoint elements $A$ and $B$ are defined by
{\small\begin{align*}
{\rm Corr}_{\varrho,\Phi}^{f,g}(A,B)&:= \Phi\left(f(\varrho) g(\varrho) AB \right)-\Phi\left(f(\varrho) A g(\varrho) B \right) \quad \text{and} \quad {\rm I}_{\varrho,\Phi}^{f,g}(A):={\rm Corr}_{\varrho,\Phi}^{f,g}(A,A)\,,
\end{align*}}
respectively. Moreover, a generalization of \eqref{skew1} in the noncommutative framework is given by
{\small\begin{equation}
\begin{bmatrix}\label{cauchcorr}
{\rm I}_{\varrho,\Phi}^{f,g}(A) & {\rm Re} \left({\rm Corr}_{\varrho,\Phi}^{f,g}(A,B)\right) \\ {\rm Re} \left({\rm Corr}_{\varrho,\Phi}^{f,g}(A,B)\right) & {\rm I}_{\varrho,\Phi}^{f,g}(B)
\end{bmatrix}\geq 0.
\end{equation}}

Here, we want to relax the linearity condition. We extend \eqref{cauchcorr} for tracial positive maps.
\begin{theorem}
 Let $\Phi : \mathscr{M} \to \mathscr{N}$ be a tracial positive map between von Neumann algebras with $\Phi(0)=0$ and let $\Phi$ enjoy the property $\mathscr{D}_m$ for some $m\geq 4$. If $\varrho$ is a density operator and ${A}, {B}\in \mathscr{M}$ are self-adjoint, then
 \[ \begin{bmatrix}
{\rm I}_{\varrho,\Phi}^{f,g}(A) & \star\\ \star & {\rm I}_{\varrho,\Phi}^{f,g}(B)
\end{bmatrix}\geq 0,
 \]
 for each pair of same monotonic functions $f$ and $g$ defined on an interval containing the spectrum of $\varrho$, where
{\small
 $$\star=\Phi\left(\frac{1}{2}f(\varrho) g(\varrho) \{A,B\} \right) - \Phi\left(\frac{f(\varrho) A g(\varrho) B +f(\varrho) B g(\varrho) A}{2} \right).$$}
 \end{theorem}
\begin{proof}
 Let $\psi: \mathscr{M} \to C(\Omega)$ and $\varphi :C(\Omega) \to \mathscr{N}$ be tracial completely positive linear and tracial $m$-positive maps decomposing $\Phi$ as in Definition \ref{decom}, respectively. Since $\psi$ is a tracial positive linear map, by applying inequality \eqref{cauchcorr} for $\psi$, we get
{\small \begin{align}\label{coor1}
\nonumber & \begin{bmatrix}
\psi\left(f(\varrho) g(\varrho) A^2 \right) & \frac{1}{2}\psi\left(f(\varrho) g(\varrho) \{A,B\} \right) \\ \frac{1}{2}\psi\left(f(\varrho) g(\varrho) \{A,B\} \right) & \psi\left(f(\varrho) g(\varrho) B^2 \right)
\end{bmatrix}\\ &\geq \begin{bmatrix}
\psi\left(f(\varrho) A g(\varrho) A \right) & \frac{1}{2}\left(\psi\left(f(\varrho) A g(\varrho) B \right) + \psi\left(f(\varrho) B g(\varrho) A \right) \right) \\ \frac{1}{2} \left( \psi\left(f(\varrho) B g(\varrho) A \right) + \psi\left(f(\varrho) A g(\varrho) B \right) \right) & \psi\left(f(\varrho) B g(\varrho) B\right)
\end{bmatrix}=: {X}.
 \end{align}}
 It is easy to see that the matrix ${X}$ to the right side of inequality \eqref{coor1} is positive. Indeed,
{\small \begin{align*}
0 &\leq \psi_2 \left(\begin{bmatrix}
f(\varrho)^\frac{1}{2} A g(\varrho)^\frac{1}{2}\\ f(\varrho)^\frac{1}{2} B g(\varrho)^\frac{1}{2}
\end{bmatrix} \begin{bmatrix}
f(\varrho)^\frac{1}{2} A g(\varrho)^\frac{1}{2}\\ f(\varrho)^\frac{1}{2} B g(\varrho)^\frac{1}{2}
\end{bmatrix}^* \right)\\ \tag{since $\psi$ is completely positive}\\ &= \psi_2 \left( \begin{bmatrix}
f(\varrho)^\frac{1}{2} A g(\varrho) A f(\varrho)^\frac{1}{2} & f(\varrho)^\frac{1}{2} A g(\varrho) B f(\varrho)^\frac{1}{2} \\ f(\varrho)^\frac{1}{2} B g(\varrho) A f(\varrho)^\frac{1}{2} & f(\varrho)^\frac{1}{2} B g(\varrho) B f(\varrho)^\frac{1}{2}
\end{bmatrix} \right)
 \\ &=\begin{bmatrix}\psi\left(f(\varrho) A g(\varrho) A \right) & \psi\left(f(\varrho) A g(\varrho) B \right) \\ \psi\left(f(\varrho) B g(\varrho) A \right) & \psi\left(f(\varrho) B g(\varrho) B\right)
\end{bmatrix}=:{Y}.
\end{align*}}
In the last equality, we employ the fact that $\psi$ is tracial.
From the positivity of the matrix ${Y}$ and the commutativity of the range of $\psi$, we infer that ${Y}^T\geq 0$, where $Y^T$ stands for the transpose of $Y$. Hence, we have
\[ {X}= \frac{1}{2} {Y} +\frac{1}{2} {Y}^T \geq 0.\]
Now, since $\varphi$ is $4$-positive (which ensures $\varphi_2$ is monotone on positive elements), by using \eqref{coor1}, we obtain
{\small \begin{align*}
& \begin{bmatrix}
 \Phi\left(f(\varrho) g(\varrho) A^2 \right) & \Phi\left(\frac{1}{2}f(\varrho) g(\varrho) \{A,B\} \right) \\ \Phi\left(\frac{1}{2}f(\varrho) g(\varrho) \{A,B\} \right) & \Phi\left(f(\varrho) g(\varrho) B^2 \right)
\end{bmatrix}\\ & \qquad =\varphi_2\left(\begin{bmatrix}
\psi\left(f(\varrho) g(\varrho) A^2 \right) & \psi\left(\frac{1}{2}f(\varrho) g(\varrho) \{A,B\} \right) \\ \psi\left(\frac{1}{2}f(\varrho) g(\varrho) \{A,B\} \right) & \psi\left(f(\varrho) g(\varrho) B^2 \right)
\end{bmatrix}\right) \\ &\qquad \geq \varphi_2\left( \begin{bmatrix}
\psi\left(f(\varrho) A g(\varrho) A \right) & \psi\left(\frac{f(\varrho) A g(\varrho) B +f(\varrho) B g(\varrho) A}{2} \right) \\\psi\left(\frac{f(\varrho) A g(\varrho) B +f(\varrho) B g(\varrho) A}{2} \right) & \psi\left(f(\varrho) B g(\varrho) B\right)
\end{bmatrix} \right) \\ &\qquad = \begin{bmatrix}
\Phi\left(f(\varrho) A g(\varrho) A \right) & \Phi\left(\frac{f(\varrho) A g(\varrho) B +f(\varrho) B g(\varrho) A}{2} \right) \\ \Phi\left(\frac{f(\varrho) A g(\varrho) B +f(\varrho) B g(\varrho) A}{2} \right) & \Phi\left(f(\varrho) B g(\varrho) B\right)
\end{bmatrix}.
 \end{align*}}
\end{proof}
 \subsection{Another view to uncertainty relations }
 In the second view, we have to consider two important points. Firstly, note that if the Hilbert spaces $\mathscr{H}_1, \ldots, \mathscr{H}_k$ are the state spaces, then the composite system is described by the tensor product Hilbert space $\mathscr{H}_1\bigotimes \cdots \bigotimes \mathscr{H}_k$; see \cite{petz}.
 Hence, if $A_1 \in \mathbb{B}(\mathscr{H}_1), \ldots, A_k\in \mathbb{B}(\mathscr{H}_k)$, then $\tilde{\textbf{A}}=A_1\otimes \cdots \otimes A_k \in \mathbb{B}(\mathscr{H}_1) \bigotimes \cdots \bigotimes \mathbb{B}(\mathscr{H}_k)$ is an observable in the composite system. On the other hand, every completely positive multilinear map $\Phi : \bigoplus_{i}^k \mathbb{B}(\mathscr{H}_i) \to \mathscr{B}$ can be extended to a positive linear map $\tilde{\Phi} : \bigotimes_{i}^k \mathbb{B}(\mathscr{H}_i) \to \mathscr{B}$ with respect to projective $C^*$-tensor product; see \cite[Corollary 1.3]{HIAI} (We remark that every completely bounded multilinear map of type 1 $\Psi : \bigoplus_{i}^k \mathbb{B}(\mathscr{H}_i) \to \mathscr{B}$ can be extended to a unique completely bounded linear map $\hat{\Psi} : \bigotimes_{i}^k \mathbb{B}(\mathscr{H}_i) \to \mathscr{B}$ with respect to the Haagerup $C^*$-norm; see \cite[Proposition 3.1]{smith}). Secondly, in the classical Heisenberg uncertainty relation
 \begin{align}\label{uncert2}
 {\rm Var}_{\tilde{\Phi}}(\tilde{\textbf{A}}) {\rm Var}_{\tilde{\Phi}}(\tilde{\textbf{B}}) &\geq \dfrac{1}{4} \left|\tilde{\Phi}\left([\tilde{\textbf{A}},\tilde{\textbf{B}}]\right)\right|^2,
 \end{align}
 if $\tilde{\textbf{A}}$ and $\tilde{\textbf{B}}$ commute, then inequality \eqref{uncert2} is trivial. In particular, if we take
 $$ \tilde{\textbf{A}}=(I_1\otimes \cdots\otimes I_i\otimes A\otimes I_{i+1}\otimes \cdots\otimes I_k) \text{\ and\ } \tilde{\textbf{B}}=(I_1\otimes \cdots\otimes I_j\otimes B\otimes I_{j+1}\otimes \cdots\otimes I_k)\quad (i\neq j),$$ as some observables in the composite system, then $\tilde{\textbf{A}}$ and $\tilde{\textbf{B}}$ commute, and the right side of inequality \eqref{uncert2} is zero. Therefore, the linear extension of completely positive multilinear maps to tensor product spaces does not provide sufficient information about the uncertainty of observables such $\tilde{\textbf{A}}$ and $\tilde{\textbf{B}}$ in composite systems. Hence, to present an uncertainty type inequality that is not trivial for such $\tilde{\textbf{A}}$ and $\tilde{\textbf{B}}$ in composite systems, we need to define a variance and a covariance that include the components of composite systems.\\
Let $ \Phi: \bigoplus_{i=1}^k \mathscr{A}_i \to \mathscr{B}$ be a unital multimap. If $1\leq i,j \leq k$, then for simplicity of notation, we define the map $\Phi_{(i,j)} : \mathscr{A}_i \bigoplus \mathscr{A}_j \to \mathscr{B}$ by
\[\Phi_{(i,j)}(A,B) = \Phi( (I_1, \ldots, I_{i-1}, A,I_{i+1},\ldots, I_k) (I_1, \ldots, I_{j-1}, B, I_{j+1},\ldots, I_k)).\] Moreover, we define the map $\Phi_{(i)} : \mathscr{A}_i \to \mathscr{B}$ by
\[\Phi_{(i)}(A)= \Phi(I_1,\ldots, I_{i-1}, A, I_{i+1}, \ldots, I_k).\]
It is clear that if $\Phi$ is a positive multilinear map, then $\Phi_{(i,j)}$ is a positive bilinear map and $\Phi_{(i)}$ is a positive linear map. \\
It is known that if $\textbf{x}= \begin{bmatrix}
x_1, & x_2, & \ldots, & x_n
\end{bmatrix}$ and $\textbf{y}= \begin{bmatrix}
y_1, & y_2, & \ldots, & y_n
\end{bmatrix}$ are two random vectors in a classical probability space, then the covariance of $\textbf{x}$ and $\textbf{y}$ is defined by the matrix
\[ {\rm Cov}(\textbf{x},\textbf{y})=\begin{bmatrix}
{\rm Cov}(x_i,y_j)
\end{bmatrix}_{i,j=1}^n.\]
From the above definition, we take the idea of the following definition.
 \begin{definition}
 Let $\Phi: \bigoplus_{i=1}^k \mathscr{A}_i \to \mathscr{B}$ be a positive unital multimap and let $\textbf{A}=(A_1,\ldots, A_k), \textbf{B}=(B_1,\ldots, B_k) \in \bigoplus_{i=1}^k \mathscr{A}_i$. Then the partial covariance of $\textbf{A}$ and $\textbf{B}$ and the partial variance of $\textbf{A}$ with respect to $\Phi$ are defined as
 \begin{align}\label{covm}
 \nonumber {\rm ^pC ov}_\Phi (\textbf{A},\textbf{B})&=\begin{bmatrix} \Phi_{(i,j)}(A_i^*, B_j)- \Phi_{(i)}( A_i^*)\Phi_{(j)} (B_j)
 \end{bmatrix}_{i,j=1}^n, \\
 {\rm ^pVar}_\Phi(\textbf{A})&= {\rm ^pCov}_\Phi (\textbf{A},\textbf{A}).
 \end{align}
 \end{definition}

We aim to prove the partial variance-covariance inequality for positive multimaps.
 We denote the partial variance-covariance matrix of two self-adjoint operators $\textbf{A}$ and $\textbf{B}$ concerning a positive map $\Phi$ by $ ^p\mathcal{V}\mathcal{C}_\Phi(\textbf{A},\textbf{B})$.\\
\begin{theorem}\label{varcovar}
Let $ \Phi: \bigoplus_{i=1}^k \mathscr{A}_i \to \mathscr{B}$ be a positive multimap between unital $C^*$-algebras with $\Phi(\textbf{0})=\textbf{0}$, and let $\textbf{A}=(A_1,\ldots, A_k), \textbf{B}=(B_1,\ldots, B_k) \in \bigoplus_{i=1}^k \mathscr{A}_i$.

\begin{itemize}
 \item[(i)] If $\Phi$ is $(2k+1)$-positive, then
\[ ^p\mathcal{V}\mathcal{C}_\Phi(\textbf{A},\textbf{B})=\begin{bmatrix}
{\rm ^pVar}_\Phi (\textbf{A}) & {\rm ^pCov}_\Phi (\textbf{A},\textbf{B}) \\
{\rm ^pCov}_\Phi (\textbf{B},\textbf{A}) & {\rm ^pVar}_\Phi (\textbf{B})
\end{bmatrix} \geq 0.\]

\item[(ii)] If $\Phi$ is $(k+1)$-positive, then ${\rm ^pVar}_\Phi(\textbf{A}) \geq 0$ for every $\textbf{A}\in \bigoplus_{i=1}^k \mathscr{A}_i$. In particular, if $\Phi$ is a positive multilinear map, then ${\rm ^pVar}_\Phi(\textbf{A}) \geq 0$ for every normal element $\textbf{A}\in \bigoplus_{i=1}^k \mathscr{A}_i$.
\end{itemize}
\end{theorem}
 \begin{proof}
 Let
 {\small \[{Z}= \begin{bmatrix}
 (A_1^*,I_2,\ldots,I_k) & 0 & \cdots & 0 \\ (I_1,A_2^*,\ldots,I_k) & 0 & \cdots & 0 \\ \vdots &\vdots & \cdots & 0 \\ (I_1, I_2,\ldots,A_k^*) & 0 & \cdots & 0
 \end{bmatrix}_{k\times k},\quad
 {W}= \begin{bmatrix}
 (B_1^*,I_2,\ldots,I_k) & 0 & \cdots & 0 \\ (I_1,B_2^*,\ldots,I_k) & 0 & \cdots & 0 \\ \vdots &\vdots & \cdots & 0 \\ (I_1, I_2,\ldots,B_k^*) & 0 & \cdots & 0
 \end{bmatrix}_{k\times k} \]}
 and
{\small
 \[\mathfrak{z}= \begin{bmatrix}
 (A_1^*,I_2,\ldots,I_k) \\ (I_1,A_2^*,\ldots,I_k) \\ \vdots \\ (I_1, I_2,\ldots,A_k^*)
 \end{bmatrix}_{k\times 1},\quad
 \mathfrak{w}= \begin{bmatrix}
 (B_1^*,I_2,\ldots,I_k) \\ (I_1,B_2^*,\ldots,I_k) \\ \vdots \\ (I_1, I_2,\ldots,B_k^*)
 \end{bmatrix}_{k\times 1}. \]}
 Then we have
 {\small
\begin{align}\label{add1}
\nonumber \begin{bmatrix} Z & & & \textbf{0}_{k\times 1}\\ W & & & \textbf{0}_{k\times 1} \\\textbf{ I} & 0 & \cdots & 0
 \end{bmatrix}_{(2k+1)\times(2k+1)}& \times \begin{bmatrix} Z^* &W^* &\textbf{I} \\
 & & 0
 \\ & & \vdots\\
\textbf{0}_{1\times k} & \textbf{0}_{1\times k} & 0
 \end{bmatrix}_{(2k+1)\times(2k+1)}
\\ &= \begin{bmatrix}
 ZZ^* & ZW^* & \mathfrak{z} \\ WZ^* & WW^* &
 \mathfrak{w} \\ \mathfrak{z}^* & \mathfrak{w}^* & \textbf{I}
 \end{bmatrix}_{(2k+1)\times(2k+1)}
 \\ \nonumber & \geq 0.
\end{align}}
If we put
 \[ {I}_{n\times n}=
 \begin{bmatrix}
\textbf{I} & 0 & \cdots & 0 \\
0 &\textbf{I} & \cdots & 0
\\ \vdots & 0 & \ddots &0
 \\ 0& 0& \cdots &\textbf{ I}
 \end{bmatrix}_{n\times n},
 \]
 then, by using the $(2k+1)$-positivity of $\Phi$ and according to the fact that $\Phi_{k-1}(\mathfrak{I}_{k-1})\geq 0$, we arrive at
 \begin{align} \label{ki4}
 \begin{bmatrix}
 \Phi_k(ZZ^*) & \Phi_k(ZW^*) & \Phi\left( \mathfrak{z} \right) \\
 \Phi_k(WZ^*) & \Phi_k(WW^*) & \Phi\left( \mathfrak{w} \right) \\
 \Phi\left(\mathfrak{z}^*\right)& \Phi\left( \mathfrak{w}^*\right) & \Phi(\mathbf{I})
 \end{bmatrix} \oplus \Phi_{k-1}({I}_{(k-1)\times (k-1)})
\geq 0,
\end{align}
in which by $\Phi\left( \mathfrak{z} \right)$ and $\Phi\left( \mathfrak{w} \right)$ we mean the column matrices
 $$
\Phi\left( \mathfrak{z} \right)=\begin{bmatrix}
 \Phi(A_1^*,I_2,\ldots,I_k)\\
 \vdots\\
 \Phi(I_1,I_2,\ldots,A_k^*)
 \end{bmatrix}_{k\times1},
 \quad\qquad
 \Phi\left( \mathfrak{w} \right)=\begin{bmatrix}
 \Phi(B_1^*,I_2,\ldots,I_k)\\
 \vdots\\
 \Phi(I_1,I_2,\ldots,B_k^*)
 \end{bmatrix}_{k\times1},$$
and $\Phi\left( \mathfrak{z}^* \right):=\Phi\left( \mathfrak{z} \right)^*$. Now note that the positive matrix appeared in \eqref{ki4} is just
 \[ \begin{bmatrix}
 \Phi_k(ZZ^*) & \Phi_k(ZW^*) & \Phi_k\left( Z \right) \\ \Phi_k(WZ^*) & \Phi_k(WW^*) & \Phi_k\left(W \right) \\ \Phi_k\left(Z^*\right)& \Phi_k\left(W^*\right) & \Phi_k({I}_{k\times k})
 \end{bmatrix}. \]
Applying Lemma \ref{epsilpo}, we reach
\begin{align}\label{cauchy}
\nonumber \begin{bmatrix}
 \Phi_k(ZZ^*) & \Phi_k(ZW^*) \\ \Phi_k(WZ^*) & \Phi_k(WW^*) \end{bmatrix}
 &\geq \begin{bmatrix}
 \Phi_k(Z)
 \\ \Phi_k(W)
 \end{bmatrix} \Phi_k({I}_{k\times k})^{-1}\begin{bmatrix}\Phi_k(Z^*) & \Phi_k(W^*)
 \end{bmatrix} \\ &= \begin{bmatrix} \Phi_k(Z) \Phi_k(Z^*) & \Phi_k(Z)\Phi_k(W^*)\\
 \Phi_k(W) \Phi_k(Z^*) & \Phi_k(W) \Phi_k(W^*)
 \end{bmatrix}.
\end{align}
Finally, the required inequality follows from \eqref{cauchy} and the fact that
\begin{align*} \begin{bmatrix}
{\rm ^pVar}_\Phi (\textbf{A}) & {\rm ^pCov}_\Phi (\textbf{A},\textbf{B}) \\
{\rm ^pCov}_\Phi (\textbf{B},\textbf{A}) & {\rm ^pVar}_\Phi (\textbf{B})
\end{bmatrix}&=\begin{bmatrix}
 \Phi_k(ZZ^*) & \Phi_k(ZW^*) \\ \Phi_k(WZ^*) & \Phi_k(WW^*) \end{bmatrix}\\ & \ \ \ - \begin{bmatrix} \Phi_k(Z) \Phi_k(Z^*) & \Phi_k(Z)\Phi_k(W^*)\\
 \Phi_k(W) \Phi_k(Z^*) & \Phi_k(W) \Phi_k(W^*)
 \end{bmatrix},\end{align*}
which proves $(\mbox{i})$.

To achieve $(\mbox{ii})$, note that if $\textbf{A}=(A_1,\ldots, A_k)$ is a normal element, then the restriction of the multilinear map $\Phi : \bigoplus_{i=1}^k \mathscr{A}_i \to \mathscr{B} $ on $C^*(\textbf{A},\textbf{I})$, is a completely positive map (so $(2k+1)$-positive). Therefore, the assertion directly follows from the positivity of $^p\mathcal{V}\mathcal{C}_\Phi(\textbf{A},\textbf{B})$. Hence, we only need to prove the first part of $(\mbox{ii})$. Let $\textbf{A}=(A_1,\ldots, A_k), \textbf{B}=(0,\ldots,0) \in \bigoplus_{i=1}^k \mathscr{A}_i$.
Then with the notations as in the proof of part $(\mbox{i})$, the matrix \eqref{add1}, after deleting the zero rows and the zero columns, turns into the positive matrix
\[ \begin{bmatrix}
 ZZ^* & \mathfrak{z} \\ \mathfrak{z}^* & \textbf{I}
 \end{bmatrix}_{(k+1)\times(k+1)}.\]
 Using the $(k+1)$-positivity of $\Phi$ and applying a similar argument as in the proof of part~$(\mbox{i})$, we reach
 \[ \begin{bmatrix}
 \Phi_k(ZZ^*) & \Phi_k\left( Z \right) \\ \Phi_k\left(Z^*\right) & \Phi_k({I}_{k\times k})
 \end{bmatrix}\geq 0.\]
Applying Lemma \ref{epsilpo}, we obtain that
\[ 0\leq \Phi_k(ZZ^*) - \Phi_k\left( Z \right) \Phi_k\left(Z^*\right)= {\rm ^pVar}_\Phi(\textbf{A}).\]
 \end{proof}
 \begin{remark}
Note that in Theorem \ref{varcovar}, if we assume that $\Phi$ is $2k$-positive, which is indeed a rather strong condition, then the positivity of the partial variance is concluded. In fact, under the notations of the proof of Theorem \ref{varcovar}, ${\rm ^pVar}_\Phi(\textbf{A}) \geq 0$ if and only if the Choi type inequality $\Phi_{k}( Z Z^*) \geq \Phi_{k}( Z) \Phi_{k}( Z^*)$ holds. Hence, if $\Phi$ is $2k$-positive, then $\Phi_k$ is $2$-positive; therefore Choi's inequality \eqref{kadison} ensures that the assertion holds true.

 Similarly, by assuming $\Phi$ to be $4k$-positive, then the positivity of matrix $^p\mathcal{V}\mathcal{C}_\Phi(\textbf{A},\textbf{B})$ is equivalent to the validity $\Phi_{2k}( {H} {H}^*) \geq \Phi_{2k}( {H}) \Phi_{2k}( {H})^*$, which is just the Choi type inequality \eqref{kadison} for $\Phi_{2k}$, where
$ {H}:=\begin{bmatrix}
 Z \\
 W
 \end{bmatrix}_{2k \times 2k}$ and $Z$ and $W$ are $k\times k$ matrices appeared in the proof of Theorem \ref{varcovar}.
 \end{remark}

 We present an uncertainty type inequality that is not trivial when we consider the commuting observables corresponding to $$\textbf{A}=(I_1, \ldots, I_i, A, I_{i+1}, \ldots, I_k), \quad \textbf{B}=(I_1, \ldots, I_j, B, I_{j+1}, \ldots, I_k)\quad (i\neq j)$$ in composite systems.
 \begin{theorem}
Let $ \Phi: \bigoplus_{i=1}^k \mathscr{A}_i \to \mathscr{B}$ be a unital positive map between unital $C^*$-algebras satisfying the property $\mathscr{D}_\infty$ and let
{\small
$$ \textbf{A}=(I_1, \ldots, I_i, A, I_{i+1}, \ldots, I_k)\quad\mbox{and}\quad\textbf{B}=(I_1, \ldots, I_j, B, I_{j+1}, \ldots, I_k)$$}
be self-adjoint elements in $\bigoplus_{i=1}^k \mathscr{A}_i$ for some $i$ and $j$ with $1\leq i\neq j \leq k$. Then
{\small\begin{align*}
 \begin{bmatrix}
{\rm Var}_\Phi(\textbf{{A}}) & 0 &\Phi( \frac{1}{2}[{\textbf{A}}, {\textbf{C}}]) & 0\\
0 & {\rm Var}_\Phi({\textbf{B}}) &0 &\Phi( \frac{1}{2}[{\textbf{B}},{\textbf{D}}])
 \\
\Phi( \frac{1}{2}[{\textbf{C}},{\textbf{A}}]) & 0 & {\rm Var}_\Phi({\textbf{C}}) & 0 \\
0 &\Phi( \frac{1}{2}[{\textbf{D}}, {\textbf{B}}]) & 0 & {\rm Var}_\Phi({\textbf{D}})
\end{bmatrix} \geq 0
\end{align*}}
for all self-adjoint elements $ \textbf{C}=(I_1, \ldots, I_i, C, I_{i+1}, \ldots, I_k)$ and $\textbf{D}=(I_1, \ldots, I_j, D,\\ I_{j+1}, \ldots, I_k)$.
In particular, if $\Phi$ is a positive multilinear map satisfying the property $\mathscr{D}_\infty$, then the above inequality turns into
{\small\begin{align*}
 \begin{bmatrix}
{\rm Var}_\Phi(\textbf{{A}}) & 0 & \frac{1}{2}\Phi([{\textbf{A}}, {\textbf{C}}]) & 0\\
0 & {\rm Var}_\Phi({\textbf{B}}) &0 &\frac{1}{2}\Phi([{\textbf{B}},{\textbf{D}}])
 \\
\frac{1}{2}\Phi([{\textbf{C}},{\textbf{A}}]) & 0 & {\rm Var}_\Phi({\textbf{C}}) & 0 \\
0 & \frac{1}{2}\Phi([{\textbf{D}}, {\textbf{B}}]) & 0 & {\rm Var}_\Phi({\textbf{D}})
\end{bmatrix} \geq 0.
\end{align*}}
\end{theorem}
\begin{proof}
Let $\psi: \bigoplus_{i=1}^k \mathscr{A}_i \to C(\mathcal{X})$ and $\varphi :C(\mathcal{X}) \to \mathscr{B}$ be the completely positive linear and completely positive maps decomposing $\Phi$ as mentioned in Definition \ref{decom}.
Set $\textbf{E}=\textbf{A} \textbf{B}$ and $\textbf{F}=\textbf{C} \textbf{D}$.
After deleting the zero rows and the zero columns of the matrix $^p\mathcal{V}\mathcal{C}_\psi(\textbf{E},\textbf{F})$, it remains the following principal submatrix of $^p\mathcal{V}\mathcal{C}_\psi(\textbf{E},\textbf{F})$
\[
\begin{bmatrix}
{X} & {Y} \\ {Y}^* & {Z}
\end{bmatrix}_{4\times 4}=:{S},
\]
whence
{\small
\[ {X}= \begin{bmatrix}
\psi_{(i)}(A^2) - \psi_{(i)}(A)^2 & \psi_{(i,j)}(A, B) - \psi_{(i)}(A) \psi_{(j)}(B)\\
 \psi_{(i,j)}(A, B) - \psi_{(j)}(B) \psi_{(i)}(A) &\psi_{(j)}(B^2) - \psi_{(j)}(B)^2
\end{bmatrix}, \]

\[ {Y}=\begin{bmatrix}
\psi_{(i)}(AC)- \psi_{(i)}(A)\psi_{(i)}(C) & \psi_{(i,j)}(A, D) - \psi_{(i)}(A) \psi_{(j)}(D) \\
\psi_{(i,j)}(C, B) - \psi_{(j)}(B) \psi_{(i)}(C) & \psi_{(j)}(BD)- \psi_{(j)}(B)\psi_{(j)}(D)
\end{bmatrix},\]}
 and
{\small
 \[ {Z}=\begin{bmatrix}
 \psi_{(i)}(C^2) - \psi_{(i)}(C)^2 & \psi_{(i,j)}(C, D) - \psi_{(i)}(C) \psi_{(j)}(D) \\ \psi_{(i,j)}(C,D)- \psi_{(j)}(D)\psi_{(i)}(C) & \psi_{(i)}(D^2) - \psi_{(i)}(D)^2
\end{bmatrix}. \]}
Since $\psi$ is completely positive, Theorem \ref{varcovar} implies that $^p\mathcal{V}\mathcal{C}_\psi(\textbf{E},\textbf{F}) \geq 0$, which ensures the positivity of ${S}$.
Put
{\small
\[{H}= \begin{bmatrix}
1 & -1 & 1 & -1 \\
-1 & 1 & -1 & 1 \\
1 & -1 & 1 & -1 \\
-1 & 1 & -1 & 1 \\
\end{bmatrix}= \begin{bmatrix}
1 \\ -1 \\ 1 \\ -1
\end{bmatrix}\begin{bmatrix}
1 & -1 & 1 & -1
\end{bmatrix}.\]}
Then ${H}$ is a positive matrix. Since the range of $\psi$ is commutative, we have
 \[
 {H} \circ {S} \geq 0,
 \]
 where ``$\circ$'' denotes the Schur product.
Both of two matrices ${H} \circ {S}$ and ${S}$ are positive. Therefore,
\[0\leq {H} \circ {S}+ {S}=\begin{bmatrix}
 {X}' & {Y}' \\ {Y}'^* & {Z}'
\end{bmatrix}_{4\times 4},\]
where
{\small
\[ {X}'= \begin{bmatrix}
2(\psi_{(i)}(A^2) - \psi_{(i)}(A)^2) & 0\\
0 & 2(\psi_{(j)}(B^2) - \psi_{(j)}(B)^2)
\end{bmatrix}, \]

\[ {Y}'=\begin{bmatrix}
2(\psi_{(i)}(AC)- \psi_{(i)}(A)\psi_{(i)}(C)) & 0 \\
0 & 2(\psi_{(j)}(BD)- \psi_{(j)}(B)\psi_{(j)}(D))
\end{bmatrix},\]}
 and
{\small
 \[ {Z}'=\begin{bmatrix}
2( \psi_{(i)}(C^2) - \psi_{(i)}(C)^2) & 0 \\ 0 & 2(\psi_{(i)}(D^2) - \psi_{(i)}(D)^2)
\end{bmatrix}. \]}
From the commutativity of the range of $\psi$ and the fact that ${Y}'$ is a diagonal matrix, we deduce that $\begin{bmatrix}
 {X}' & {Y}'^*\\ {Y}'& {Z}'
\end{bmatrix}_{4\times 4}\geq 0$. Furthermore,
 \[ \begin{bmatrix}
 {X}' & - {Y}'^*\\ - {Y}' &{Z}'
\end{bmatrix}_{4\times 4}= \begin{bmatrix}
1 \\ 1 \\ -1 \\ -1
\end{bmatrix}\begin{bmatrix}
 {X}' & {Y}'^*\\ {Y}' & {Z}'
\end{bmatrix}_{4\times 4}\begin{bmatrix}
1 & 1 & -1 & -1
\end{bmatrix}\geq 0.\]
Hence,
\begin{footnotesize}\begin{align}
 0 &\leq\begin{bmatrix}\label{matrixenddd}
 {X}' & {Y}' \\ {Y}'^* & {Z}'
\end{bmatrix}_{4\times 4} + \begin{bmatrix}
 {X}' & - {Y}'^*\\ - {Y}' & \mathfrak{Z}'
\end{bmatrix}_{4\times 4}
\\ \nonumber &= \begin{bmatrix}
4(\psi_{(i)}(A^2) - \psi_{(i)}(A)^2) & 0 & 2\psi_{(i)}([A,C]) &0\\
0 & 2(\psi_{(j)}(B^2) - \psi_{(j)}(B)^2) & 0 & 2\psi_{(j)}([B,D])\\
2\psi_{(i)}([C,A]) &0 & 4( \psi_{(i)}(C^2) - \psi_{(i)}(C)^2) & 0 \\
0 & 2\psi_{(j)}([D,B]) & 0 & 4(\psi_{(i)}(D^2) - \psi_{(i)}(D)^2)
\end{bmatrix} \\ \tag{since $\psi$ is linear and its range is commutative}
 \\ & =
 \begin{bmatrix}\label{matrixend}
4 {\rm Var}_\psi(\textbf{{A}}) & 0 & 2\psi([{\textbf{A}}, {\textbf{C}}]) & 0\\
0 & 4{\rm Var}_\psi({\textbf{B}}) &0 &2\psi([{\textbf{B}},{\textbf{D}}])
 \\
2\psi([{\textbf{C}},{\textbf{A}}]) & 0 & 4{\rm Var}_\psi({\textbf{C}}) & 0 \\
0 & 2 \psi([{\textbf{D}}, {\textbf{B}}]) & 0 & 4 {\rm Var}_\psi({\textbf{D}})
\end{bmatrix}.
\end{align}\end{footnotesize}
Utilizing a similar argument as in the proof of inequality \eqref{varicomp}, it can be easily shown that
\begin{equation}\label{varicomp2}
\varphi \left({\rm Var}_{\psi}( \textbf{X})\right) \leq {\rm Var}_{\Phi}\left(\textbf{X}\right)
\end{equation}
for every $\textbf{X}\in \bigoplus_{i=1}^k \mathscr{A}_i$.
Hence, by applying the positivity of matrix \eqref{matrixend} and using the complete positivity of $\varphi$, we see that
{\small\begin{align*}
0 &\leq \varphi_4\left( \begin{bmatrix}
{\rm Var}_\psi(\textbf{{A}}) & 0 & \psi(\frac{1}{2}[{\textbf{A}}, {\textbf{C}}]) & 0\\
0 & {\rm Var}_\psi({\textbf{B}}) &0 &\psi(\frac{1}{2}[{\textbf{B}},{\textbf{D}}])
 \\
\psi(\frac{1}{2}[{\textbf{C}},{\textbf{A}}]) & 0 & {\rm Var}_\psi({\textbf{C}}) & 0 \\
0 & \psi(\frac{1}{2}[{\textbf{D}}, {\textbf{B}}]) & 0 & {\rm Var}_\psi({\textbf{D}})
\end{bmatrix}\right)\\ &=
\begin{bmatrix}
 \varphi\left({\rm Var}_\psi(\textbf{{A}})\right) & 0 & \Phi \left(\frac{1}{2}[{\textbf{A}}, {\textbf{C}}]\right) & 0\\
0 & \varphi\left({\rm Var}_\psi({\textbf{B}})\right) &0 &\Phi \left(\frac{1}{2}[{\textbf{B}},{\textbf{D}}]\right)
 \\ \Phi\left(\frac{1}{2}[{\textbf{C}},{\textbf{A}}]\right) & 0 & \varphi\left({\rm Var}_\psi(\textbf{{C}})\right) & 0 \\
0 & \Phi\left(\frac{1}{2}[{\textbf{D}}, {\textbf{B}}]\right) & 0 & \varphi\left({\rm Var}_\psi(\textbf{{D}})\right)
\end{bmatrix}
\\ & \leq \begin{bmatrix}
 {\rm Var}_\Phi(\textbf{{A}}) & 0 & \Phi \left(\frac{1}{2}[{\textbf{A}}, {\textbf{C}}]\right) & 0\\
0 & {\rm Var}_\Phi({\textbf{B}}) &0 &\Phi \left(\frac{1}{2}[{\textbf{B}},{\textbf{D}}]\right)
 \\ \Phi\left(\frac{1}{2}[{\textbf{C}},{\textbf{A}}]\right) & 0 & {\rm Var}_\Phi(\textbf{{C}}) & 0 \\
0 & \Phi\left(\frac{1}{2}[{\textbf{D}}, {\textbf{B}}]\right) & 0 & {\rm Var}_\Phi(\textbf{{D}})
\end{bmatrix}, \tag{by inequality $\eqref{varicomp2}$}
\end{align*}}
which proves the first inequality. If $\Phi$ is a multilinear map, then the maps $\Phi_{(i)}$ and $\Phi_{(j)}$ are linear, and the positivity of matrix \eqref{matrixenddd} implies that
\begin{footnotesize}
\[
 \begin{bmatrix}\label{matrixend0}
\psi_{(i)}(A^2) - \psi_{(i)}(A)^2 & 0 & \psi_{(i)}(\frac{AC-CA}{2}) &0\\
0 & \psi_{(j)}(B^2) - \psi_{(j)}(B)^2 & 0 & \psi_{(j)}\left(\frac{BD-DB}{2}\right)\\
\psi_{(i)}(\frac{CA-AC}{2}) &0 & \psi_{(i)}(C^2) - \psi_{(i)}(C)^2 & 0 \\
0 & \psi_{(j)}(\frac{DB-BD}{2}) & 0 & \psi_{(i)}(D^2) - \psi_{(i)}(D)^2
\end{bmatrix} \geq 0.
\]\end{footnotesize}
Put $\textbf{X}=( I_1, \ldots, I_{s-1}, {X}, I_s, \ldots, I_k)$. Since $\varphi_2$ is multilinear, we have
{\small\begin{align*}
\varphi_2\circ \psi_{(s)}(\frac{X}{2})&=\varphi_2\circ \psi( I_1, \ldots, I_{s-1}, \frac{X}{2}, I_s, \ldots, I_k)\\
&= \frac{1}{2} \Phi( I_1, \ldots, I_{s-1}, {X}, I_s, \ldots, I_k)=\frac{1}{2} \Phi(\textbf{X}).
\end{align*}}
Now, by the same reasoning as in the first part, we reach the second assertion.
\end{proof}
The positivity of the second matrix in the above theorem immediately gives the following corollary.
\begin{corollary}\label{enddd}
Let $ \Phi: \bigoplus_{i=1}^k \mathscr{A}_i \to \mathscr{B}$ be a multilinear map between $C^*$-algebras satisfying the property $\mathscr{D}_\infty$. Let
$$ \textbf{A}=(I_1, \ldots, I_i, A, I_{i+1}, \ldots, I_k),\qquad \textbf{C}=(I_1, \ldots, I_i, C, I_{i+1}, \ldots, I_k)$$
and
$$\textbf{B}=(I_1, \ldots, I_j, B, I_{j+1}, \ldots, I_k),\qquad \textbf{D}=(I_1, \ldots, I_j, D, I_{j+1}, \ldots, I_k)$$
be self-adjoint operators in $\bigoplus_{i=1}^k \mathscr{A}_i$ with $1\leq i\neq j \leq k$. Then
\begin{enumerate}
\item
 {\small\[ \|{\rm Var}_\Phi(\textbf{{A}})\| \|{\rm Var}_\Phi({\textbf{B}})\| \|{\rm Var}_\Phi({\textbf{C}})\| \|{\rm Var}_\Phi({\textbf{D}})\| \geq \frac{1}{16} \left\|\Phi([{\textbf{A}}, {\textbf{C}}])\right\|^2 \left\|\Phi([{\textbf{B}},{\textbf{D}}])\right\|^2,\]}
 \item if the range of $\Phi$ is commutative, then
{\small
\[ {\rm Var}_\Phi(\textbf{{A}}) {\rm Var}_\Phi({\textbf{B}}) {\rm Var}_\Phi({\textbf{C}}) {\rm Var}_\Phi({\textbf{D}}) \geq \frac{1}{16} \left|\Phi([{\textbf{A}}, {\textbf{C}}])\right|^2 \left|\Phi([{\textbf{B}},{\textbf{D}}])\right|^2.\]}
\end{enumerate}
\end{corollary}
It follows from \cite[Theorem 3.1]{shar} that if $\Phi: \mathscr{A}\to \mathscr{B}$ is a positive linear map between $C^*$-algebras, then there is an upper bound for the variance as the following form:
\[ {\rm Var}_\Phi(X) \leq \inf_{\alpha \in \mathbb{C}} \|{X}-\alpha I\|^2 \qquad (X\in \mathscr{A}). \]
Considering the map $\Psi_{(s)}$, it is easy to check that for every $\textbf{X}=( I_1, \ldots,\\ I_{s-1}, {X}, I_{s+1}, \ldots, I_k)\in \bigoplus_{i=1}^k \mathscr{A}_i$ and any unital positive multilinear map $\Psi: \bigoplus_{i=1}^k \mathscr{A}_i\to \mathscr{B}$ between unital $C^*$-algebras, it holds that
\[ {\rm Var}_\Psi(\textbf{X}) \leq \inf_{\alpha \in \mathbb{C}} \|{X}-\alpha I_s\|^2. \]
In virtue of \cite[Theorem 4 and Corollary 1]{der}, if $X$ is a normal element, then $\inf_{\alpha \in \mathbb{C}} \|X-\alpha I\|$ is equal to the radius of the smallest disk containing the spectrum of $X$. In particular, if $X$ is a positive element or $X$ is a positive contraction, then $\inf_{\alpha \in \mathbb{C}} \|X-\alpha I\|\leq \|X\| $ or $\inf_{\alpha \in \mathbb{C}} \|X-\alpha I\|\leq \frac{1}{2}$, respectively. \\
Thus, from Corollary \ref{enddd} and the correspondence between completely positive multilinear maps and completely positive linear maps on tensor product spaces \cite[Corollary 1.3]{HIAI}, the following conclusion can be expressed for an uncertainty relation of two commuting observables.
\begin{corollary}
Let $ \Phi: \mathscr{A} \bigotimes \mathscr{B} \to \mathscr{C}$ be a completely positive linear map between $C^*$-algebras with the commutative range and let $ A\otimes I_\mathscr{B}$ and $I_\mathscr{A}\otimes B $ be compatible observables in the composite system $ \mathscr{A} \bigotimes \mathscr{B}$ (with respect to the projective tensor product). Then the following statements hold:
\begin{enumerate}
\item For all self-adjoint elements $ C\in \mathscr{A}$ and $D\in\mathscr{B}$ and all $\alpha, \beta \in \mathbb{C}$ with nonzero $ \|C-\alpha I_\mathscr{A}\|$ and $ \|D-\beta I_\mathscr{B}\| $, it holds that
{\small
\[ {\rm Var}_\Phi(A\otimes I_\mathscr{B}) {\rm Var}_\Phi(I_\mathscr{A}\otimes B) \geq \frac{1}{16} \frac{ \left|\Phi([A, C]) \Phi([B, D])\right|^2}{(\|C-\alpha I_\mathscr{A}\| \|D-\beta I_\mathscr{B}\|)^2 }.\]}
\item If $C$ and $D$ are positive contractions, then the above inequality turns into
{\small
\[ {\rm Var}_\Phi(A\otimes I_\mathscr{B}) {\rm Var}_\Phi(I_\mathscr{A}\otimes B) \geq \left|\Phi([A, C])\right|^2 \left|\Phi([B, D])\right|^2.\]}
\end{enumerate}
\end{corollary}

\medskip

\section{Declarations}

\noindent \textit{Ethical Approval.} Ethical approval is not applicable to this article as the research is not related to human and/or animal studies.\\

\noindent \textit{Authors' contributions.} The initial idea of the work was given by M. S. Moslehian. All authors have participated in investigating and obtaining the results. Then, the first draft of the manuscript was written by A. Dadkhah and the final version was prepared by M. Kian. All authors commented on the final versions of the manuscript. The manuscript was edited and approved for submission by all authors.\\

\noindent \textit{Funding.} Funding information is not applicable yet.\\

\noindent \textit{Conflict of Interest Statement.} On behalf of all authors, the corresponding author states that there is no conflict of interest.\\

\noindent\textit{Data Availability Statement.} Data sharing not applicable to this article as no datasets were generated or analysed during the current study.

\medskip

\bibliographystyle{amsplain}

\end{document}